\documentclass[10pt, a4paper]{amsart}
\usepackage{article}

\title[Dynamical coherence and Shadowing Lemma]{Partially hyperbolic diffeomorphisms 
with a uniformly compact center foliation: the quotient dynamics}
\author[Bohnet]{Doris Bohnet}
\address{Universit\'e de Bourgogne, France}
\email{Doris.Bohnet@u-bourgogne.fr}
\thanks{This paper was partially supported by the \emph{Forschungsfond} 
of the Department of Mathematics and the 
\emph{K\"{o}rperschaftsverm\"{o}gen} from the Universit\"at Hamburg. 
It was further supported by the Universit\'e de Bourgogne, 
in particular the IMB, and the first author would like to thank the 
\emph{Institut de Math\'ematiques de Bourgogne} for its hospitality during her several visits. 
Moreover, the first author profits from a \emph{Bourse Postdoc} given by the \emph{R\'egion Bourgogne} 
which helped considerably to finish this article. }
\author[Bonatti]{Christian Bonatti}
\address{Universit\'e de Bourgogne, France}
\email{bonatti@u-bourgogne.fr}

\date{\today}

\begin{document}

\begin{abstract}
We show that a partially hyperbolic $C^1$-diffeomorphism $f:M~\rightarrow~M$ with a
uniformly compact $f$-invariant center foliation $\mathcal{F}^c$ is dynamically coherent. 
Further, the induced homeomorphism $F:M/\mathcal{F}^c~\rightarrow~M/\mathcal{F}^c$ on the 
quotient space of the center foliation has the shadowing property, i.e. for every 
$\epsilon~>~0$ there exists $\delta~>~0$ such that every $\delta$-pseudo orbit of center leaves 
is $\epsilon$-shadowed by an orbit of center leaves. Although the shadowing orbit is not necessarily unique, 
we prove the density of periodic center leaves inside the chain recurrent set of the quotient dynamics. Some other interesting properties of 
the quotient dynamics are discussed.  
\end{abstract}
\maketitle 
%\tableofcontents

\section{Introduction}
Among the chaotic dynamical systems, the hyperbolic diffeomorphisms are the one whose dynamics 
is now well understood. In particular, they are structurally stable: the nearby diffeomorphisms 
are conjugate, hence, they have the same topological dynamics. This leads to the hope that one
could classify their dynamics, up to conjugacy. The \emph{shadowing lemma} is at the same time a very useful tool 
for understanding the dynamics of hyperbolic systems (for instance for proving density 
of periodic orbits in the chain recurrent set) or for proving the structural stability: 
the orbits of a nearby system are pseudo orbits of the initial system, hence are shadowed
by true orbits. 

Here we consider partially hyperbolic systems: a $C^1$-diffeomorphism $f:M~\rightarrow~M$ on a
smooth compact manifold $M$ is called \emph{partially hyperbolic} if its tangent bundle splits 
into three non-trivial, $df$-invariant subbundles, called the stable, unstable and center bundle
$$TM=E^s\oplus E^c\oplus E^u,$$
such that $df$ contracts uniformly vectors in the stable direction, expands uniformly vectors in 
the unstable direction and contracts and/ or expands in a weaker way vectors in the center direction. 
As hyperbolicity, partial hyperbolicity is a $C^1$-robust property.
 
None of these bundles is \emph{a priori} differentiable. However, dynamical arguments ensure that the 
stable and unstable bundles integrate to unique $f$-invariant stable and unstable foliations. 
In contrast,  the center bundle might not be integrable, even in a weak sense (see \cite{W98,HHU10a}). 
If there is a foliation tangent to the center bundle, it is not known if 
it is  unique, and therefore it is not known if there exists an invariant center foliation. Finally if such an invariant center foliation exists, 
it may exhibit pathological features: for instance, it may not be absolutely continuous with respect to 
Lebesgue measure, even for analytic partially hyperbolic diffeomorphisms (see \cite{SW00,G12}).

However, if the center foliation exists and satisfies a condition called \emph{plaque expansivity}, 
\cite{HPS70} proved that it is \emph{$C^1$-structurally stable}: every diffeomorphism $g$  $C^1$-close 
to $f$ admits a center foliation $\cF^c_g$ conjugate to $\cF^c_f$ by a homeomorphism $h$ so that
$hgh^{-1}$  is a $C^0$-perturbation of $f$ along the center leaves. In other words, the dynamics 
\emph{transverse to the center bundle} is structurally stable. 
This leads to the hope of a classification of this transverse dynamics, up to topological 
conjugacy. 

\vskip 2mm
Here we consider a natural setting with an assumption which could seem very restrictive: 
we restrict ourselves to the 
class of partially hyperbolic diffeomorphisms where there is an invariant center foliation whose 
leaves are 
all compact. This class contains all the 
partially hyperbolic skew products over an Anosov diffeomorphism. Even assuming that every center 
leaf is compact, the center foliation may not be a fibration: there are examples of 
partially hyperbolic 
diffeomorphisms where the center foliation is a (generalized) Seifert bundle, so there are 
leaves with 
non-trivial holonomy (see \cite{BoW05}). Examples are partially hyperbolic automorphisms of the Heisenberg manifold. It is natural to ask if a center foliation 
by compact leaves  may be even more complicated. This question makes sense as Sullivan 
 builds an example of a circle foliation on a compact manifold for 
which the length of the leaves is not bounded (\cite{S76a}, see also \cite{EV78}). 
According to \cite{E76} this is equivalent 
to the fact that some leaf has a non finite holonomy group. 
Therefore, we  call \emph{uniformly compact} the 
foliations by compact leaves with finite holonomy, and we consider diffeomorphisms for which the 
center foliation is uniformly compact. It  is conjectured 
that every compact 
center foliation is uniformly compact, and there are  results 
by Gogolev \cite{G11} and Carrasco 
\cite{C11}   proving this conjecture under additional assumptions.

In the uniformly compact case \cite{E76} shows that the quotient space $M/\cF^c$ is a compact 
metric space with respect 
to the Hausdorff distance between the center leaves. The invariance of the center foliation $\cF^c$ means 
that the diffeomorphism $f$ passes in the quotient to a homeomorphism $F$. The aim of this paper is 
the study of this quotient dynamics. 

We show that the normal hyperbolicity of the center foliation in $M$ implies some kind of 
topological hyperbolic behavior for the quotient dynamics.  As a first step, we would like to project 
onto $M/\cF^c$ the stable and  unstable foliations of $f$. Such a projection requires some compatibility 
between the stable, the unstable and the center foliation, called \emph{dynamical coherence}: we need 
the existence of invariant foliations $\cF^{cs}$ and $\cF^{cu}$
tangent to $E^{cs}=E^s\oplus E^c$ and $E^{cu}=E^u\oplus E^c$, respectively, and subfoliated by $\cF^c$, 
and $\cF^s$ and $\cF^u$, respectively. It should be mentioned that there exist non-dynamically coherent 
partially hyperbolic systems, even on the $3$-torus (see \cite{HHU10a}). 
This is our first result: 

\begin{thm}[Dynamical coherence] Let $f$ be a partially hyperbolic diffeomorphism on a compact manifold 
admitting an invariant uniformly compact center foliation $\cF^c$. Then $(f,\cF^c)$ is dynamically 
coherent. 
\label{theorem_dc}
\end{thm}

For proving Theorem~\ref{theorem_dc}, one shows that the center leaf  $W^c(y)$ through a point 
$y\in W^s(x)$  is contained in the union of stable leaves through the points 
$z\in W^c(x)$.  An intuitive idea is that, as $y\in W^s(x)$ and as the center foliation is 
uniformly compact, one easily deduces that the Hausdorff distance $d_H(f^n(W^c(x)), f^n(W^c(y)))$ 
tends to $0$ for $n \rightarrow \infty$, leading to the intuition that $W^c(y)$ is contained in 
the stable manifolds through $W^c(x)$.  However, this argument is not sufficient: we  exhibit in 
Proposition~\ref{p.example} a partially hyperbolic diffeomorphism with a uniformly compact invariant 
center foliation having two center leaves $W^c(x)$ and $W^c(y)$ so that the Hausdorff distance 
$d_H(W^c(x),W^c(y))$ tends to $0$ under positive iterates, but $W^c(y)$ is disjoint from 
$\bigcup_{z\in W^c(x)}W^s(z)$. The intuitive argument above is indeed correct if the center foliation is a locally 
constant fibration (that is, without holonomy); in the general case (with holonomy) we modify the distance 
by considering lifts of the leaves on holonomy covers. Such holonomy covers only exist locally, 
leading to many difficulties and
a somewhat technical notion of distance, which is one of the key points of this paper.

The dynamical coherence implies that the quotient space is endowed 
with the quotient of the center stable and center unstable foliations. We would like to
use the quotient of these foliations for proving hyperbolic properties of the quotient dynamics.
In particular, we aim to recover the \emph{shadowing property} (every pseudo orbit is shadowed by a true orbit)  
which is one of 
the main characteristics of hyperbolicity and an important tool for describing the topological dynamics: 
 for systems satisfying the shadowing property, the non-wandering 
set coincides with the chain recurrent set; therefore, one may apply Conley theory 
(dividing the chain recurrent set in chain recurrent classes separated by filtrations) 
to the non-wandering set.   Furthermore, in the case of hyperbolic systems, the system is \emph{expansive}: 
if the orbits of two points remain close for all time, the two points are equal; 
thus, the shadowing orbit is unique. This uniqueness leads to an even 
more precise description of the dynamics, e.g. the density of the periodic orbits in the non-wandering set.

In the hyperbolic case, the proof of the shadowing property uses strongly the stable and unstable foliations and
the fact that they form a local product structure. 
However, in our setting  the quotient on $M/\mathcal{F}^c$ of the center stable and center unstable foliations 
may not be foliations, as 
in \cite[Proposition 4.4]{BoW05}, and they may not induce a local product 
structure on the quotient. In particular, the quotient dynamics may not be expansive, 
as in the  example in \cite{BoW05}. This induces a difficulty for recovering the Shadowing Lemma: 
the shadowing orbit may not be unique, if it exists.  Nevertheless, we prove:\footnote{At the moment of submission we learnt that Kryzhevich and Tikhomirov announce in \cite{KT12} a Shadowing Lemma possibly related to ours. Their setting is for one side more general as they do not assume a compact 
foliation but on the other side they suppose a uniquely integrable center and 
a stronger version of partial hyperbolicity.}

\begin{thm}[Shadowing Lemma]
Let $f:M~\rightarrow~M$ be a partially hyperbolic $C^1$-diffeomorphism with 
a uniformly compact invariant center foliation $\mathcal{F}^c$. Then the induced
homeomorphism $F:M/\mathcal{F}^c~\rightarrow~M/\mathcal{F}^c$ has the shadowing property, i.e. for every 
$\epsilon~>~0$ there exists $\delta~>~0$ such that every $\delta$-pseudo orbit of center leaves 
is $\epsilon$-shadowed by an orbit of center leaves. 
\label{shadowing_lemma}
\end{thm}

In Theorem~\ref{shadowing_lemma}, infinitely many orbits may shadow the same 
pseudo-orbit (see the example in Corollary~\ref{c.example}).
The non-uniqueness of the shadowing  makes that the shadowing property 
does not imply the existence of periodic orbits 
(here periodic center leaves). The lack of uniqueness comes from possible choices during the construction 
caused by the non-trivial holonomy. Theorem~\ref{t.unique} expresses that we recover the 
uniqueness of the shadowing orbit if we enrich the pseudo orbits by these possible choices.  Therefore, we get:

\begin{thm}\label{t.periodic}
 Let $f:M~\rightarrow~M$ be a partially hyperbolic $C^1$-diffeomorphism with 
a uniformly compact invariant center foliation $\mathcal{F}^c$. Then the periodic center leaves are dense 
in the chain recurrent set of the induced
homeomorphism \newline$F:M/\mathcal{F}^c~\rightarrow~M/\mathcal{F}^c$. In particular, the chain recurrent set 
coincides with the non-wandering set and with the closure of the set of periodic points.

\end{thm}
Theorem~\ref{t.periodic} implies that, if the chain recurrent set (or non-wandering set) 
of the quotient dynamics $F$ 
is the whole $M/\mathcal{F}^c$, then $F$ is transitive. This is an important tool in \cite{B12} 
for classifying codimension $1$ partially hyperbolic diffeomorphisms with a uniformly compact center 
foliation.

In the case of hyperbolic dynamics, the expansivity is a key ingredient  for the structural stability. In the case of 
partially hyperbolic system with center foliation, this has been adapted in \cite{HPS70} to the notion of 
\emph{plaque expansivity}: the center foliation is plaque expansive if every 
two pseudo-orbits respecting the center foliation (i.e. with jumps in the local center leaves) remain close for 
all time, then they lie in the same  center leaf. As said before, \cite{HPS70} proves the 
structural stability of the center foliation for partially hyperbolic systems with plaque 
expansive center foliation.

Even, if the quotient dynamics is not expansive, the proof of the Shadowing Lemma uses some 
expansiveness property, which implies in particular the plaque expansivity: 

\begin{thm}[Plaque expansivity]
Let $f:M~\rightarrow~M$ be a partially hyperbolic $C^1$-diffeomorphism with a uniformly 
compact invariant center foliation $\mathcal{F}^c$. Then $(f,\cF^c)$ is plaque expansive. 
\label{plaquex}
\end{thm}
\begin{corol}[Structural stability] Let $f:M~\rightarrow~M$ be a partially hyperbolic $C^1$-diffeomorphism with a 
uniformly compact invariant center foliation $\mathcal{F}^c$. Then $(f,\cF^c)$ is $C^1$-structurally 
stable (in the sense of \cite{HPS70}).
\end{corol}
In particular, the regularity of the diffeomorphism $f$ does not interfere with their topological 
classification.
\begin{rem}
Theorem~\ref{theorem_dc} and Theorem~\ref{shadowing_lemma} have already appeared in \cite{B11}, but were not properly stressed as important results by themselves. Here they are embedded in a diligent and more complete study of the quotient dynamics. 
\end{rem}
\textbf{This work is organized as follows:}
In Section~\ref{s.background} we recall the relevant notions and implications from foliation theory 
with respect to uniformly compact foliations, and we develop in Section~\ref{s:covers} a concept of 
holonomy cover which suits our purposes. We prove dynamical coherence (Theorem~\ref{theorem_dc}) in 
Section~\ref{sec:coherence} in several steps, including the proof of the existence of a well-defined 
notion of unstable projection of a center leaf. All these preparatory steps help to prove the 
Shadowing Lemma (Theorem~\ref{shadowing_lemma}) in Section~\ref{sec:shadowing} quite directly, 
following the classical proof. The subsequent Section~\ref{sec:plaque} is reserved for the proof 
of the plaque expansivity (Theorem~\ref{plaquex}), followed by the very direct implication, 
Theorem~\ref{theorem_noncompact} (Section~\ref{sec:noncompact}), of non-compactness of 
center-stable and center-unstable foliations given a uniformly compact center foliation. 
We end this article with a short discussion in Secion~\ref{sec:comments} of an example (presented in \cite{BoW05}) which stresses the important differences of the 
quotient dynamics to the classical hyperbolic behavior.  

\renewcommand{\abstractname}{Acknowledgements}
\begin{abstract}
 We would like to thank the referee for his careful reading and his numerous 
comments which helped us to
improve the presentation of this article.
\end{abstract}

\newpage
\section{Preliminaries from foliation theory}\label{s.background}

This section recalls the basic concepts of foliation theory used in this article.

A  \emph{$C^{1,0+}$-foliation} is a foliation on a manifold $M$ for which the leaves are $C^1$-immersed 
submanifolds and the solution of a distribution defined by a continuous subbundle of $TM$. 
Throughout this article the word foliation refers to \emph{$C^{1,0+}$-foliation}.

\subsection{Holonomy}

We shortly define the notion of holonomy for a leaf of a foliation. 
For a more complete definition we refer the reader to the books 
by Candel and Conlon (\cite{CC00}) and by Moerdijk and Mrcun (\cite{MM03}). See Figure~\ref{fig:holonomy}. \\
Consider a  foliation $\cF$ on a manifold $M$ and a closed path $\gamma:[0,1]~\rightarrow~L$ with $\gamma(0)=\gamma(1)=x \in L$ which lies entirely 
inside one leaf $L \in \cF$.
We define a homeomorphism $H_{\gamma}$ on a smooth disk $T$ of dimension $q=\codim \mathcal{F}$ transversely 
embedded to the foliation $\mathcal{F}$ at $x$ which fixes $x$ and maps intersection points of a nearby leaf $L'$ 
onto each other following the path $\gamma$. We call it the \textit{holonomy homeomorphism along the path $\gamma$}. The definition of $H_{\gamma}$ (more precisely of the class of germs of $H_{\gamma}$) only depends on the homotopy class $[\gamma]$ of $\gamma$. Hence, we obtain a group homomorphism 
$$
\pi_1\left(L,x\right) \rightarrow \Homeo\left(\mathbb{R}^q,0\right),
$$
where $\Homeo(\mathbb{R}^q,0)$ denotes the classes of germs of homeomorphisms 
$\mathbb{R}^q \rightarrow \mathbb{R}^q$ which fix the origin.
The image of this group homomorphism is called the \textit{holonomy group} of the leaf 
$L$ and denoted by $\Hol(L,x)$. 
By taking the isomorphism class of this group it does not depend on the original embedded disk
$T$ in $M$. It is easily seen that any simply connected leaf has a trivial holonomy group. We say that a 
leaf has \textit{finite holonomy} if the holonomy group $\Hol(L,x)$ for any $x \in L$ is a finite group, 
whose order is denoted by $|\Hol(L,x)|$. 
A foliation is  \emph{uniformly compact} if every leaf is compact and has finite holonomy.
This is the main object we consider in the following.

A subset of the manifold  is called \emph{$\cF$-saturated}   if it is a union of leaves.
\begin{figure}
\includegraphics[width=0.5\textwidth]{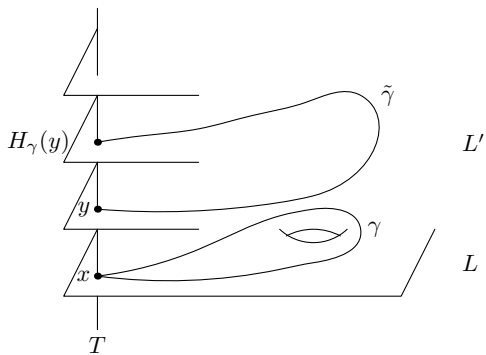}
\caption{Construction of a holonomy homeomorphism $H_{\gamma}: T \rightarrow T$: It fixes $x=\gamma(0)=\gamma(1)$ and maps $y=\tilde{\gamma}(0)$ onto $\tilde{\gamma}(1)$ where $\tilde{\gamma}$ denotes the lift of $\gamma$ to the nearby leaf $L'$.}
\label{fig:holonomy}
\end{figure}

%In codimension one (also on a non-compact manifold) it is shown by Reeb in \cite{R52} and in codimension two on compact manifolds by Epstein in %\cite{E72}, that any compact foliation has finite holonomy. But in higher codimensions there exist counterexamples: Sullivan constructed in %\cite{S76a} and \cite{S76b} a flow on a compact $5$-manifold such that every orbit is periodic but the length of orbits is unbounded. Later, Epstein %and Vogt in \cite{EV78} constructed an example with the same properties on a compact $4$-manifold.  In the following we will therefore assume a %compact center foliation with finite holonomy (if the codimension is greater than $2$).\\

\subsection{Uniformly compact foliations.}    
A classical result (proved in \cite{H77} and \cite{EMT77}) states that leaves with trivial holonomy groups 
are generic. In the case of a foliation whose leaves are all compact this implies that there 
exists an open and dense set of leaves with trivial holonomy.  
The first reason why we assume a uniformly compact foliation are the following equivalences: 
\begin{otherthm}[Epstein (\cite{E76})]
Let $\cF$ be a foliation with all leaves compact of a manifold $M$. Then the following conditions are equivalent: 
\begin{itemize}
 \item The quotient map $\pi:M~\rightarrow~M/\mathcal{F}$ is closed. 
\item Each leaf has arbitrarily small saturated neighborhoods. 
\item The leaf space $M/\mathcal{F}$ is Hausdorff. 
\item If $K~\subset~M$ is compact then the saturation $\pi^{-1}\pi K$ of $K$ is compact, this means, the set of leaves meeting a compact set is compact.
 \item The holonomy group of every leaf is finite.  
\end{itemize}
\label{theorem_epstein}
\end{otherthm}
\begin{rem}Furthermore, the quotient topology is generated by the Hausdorff metric $d_H$ between center leaves in $M$ and hence, $M/\mathcal{F}^c$ is a compact metric space. In the case of a compact foliation with trivial holonomy the resulting leaf space is a topological manifold. 
\end{rem}

\textbf{Notation:} For every leaf $L \in \mathcal{F}$ and $\delta~>~0$, we denote by $B_H(L,\delta)$ the $\mathcal{F}$-saturated open ball of Hausdorff radius  $\delta$ around $L$ :
$$B_H(L,\delta):=\left\{L' \in \mathcal{F}\;\big|\; d_H(L,L') < \delta\right\}.$$

According to Theorem \ref{theorem_epstein}, every leaf admits a basis of  
saturated neighborhoods. As a consequence of this and of the compactness of the leaf space, one gets: 
\begin{corol}
Let $\mathcal{F}$ be a uniformly compact foliation of a compact manifold $M$. Then for given $\alpha~>~0$ there exists $\epsilon~>~0$ such that for every $x~\in~M$ the $\mathcal{F}$-saturate $\mathcal{F}\left(B\left(x,\epsilon\right)\right)=\bigcup_{y \in B(x,\epsilon)}L(y)$ of an $\epsilon$-ball $B(x,\epsilon)$ is contained in an $\mathcal{F}$-saturated Hausdorff neighborhood $B_H(L(x),\alpha)$ of the leaf $L(x)$, i.e.  
$$\mathcal{F}\left(B\left(x,\epsilon\right)\right) \subset B_H(L(x),\alpha).$$
\label{lemma_neighborhoods}
\end{corol}

\subsection{Neighborhood of a compact leaf}

An important property of uniformly compact foliations is that there exist small saturated 
neighborhoods $V(L)$ of any leaf $L$ that are foliated bundles $p:V(L)~\rightarrow~L$ and the holonomy group is actually a group. This fact is the content of the Reeb Stability Theorem which is proved in the present form in \cite[Theorem 2.4.3]{CC00}:    
\begin{otherthm}[Generalized Reeb Stability]
Let $L$ be a compact leaf of a foliation $\cF$ of a manifold $M$. 
If its holonomy group $\Hol(L)$ is finite, then there is a normal neighborhood 
$p:V~\rightarrow~L$ of $L$ in $M$ such that $\left(V, \mathcal{F}|_V,p\right)$ is a fiber bundle with 
finite structure group $\Hol(L)$. 

Furthermore, each leaf $L'|_V$ is a covering space $p|_{L'}: L'~\rightarrow~L$ with $k~\leq~\left|\Hol(L,y)\right|$ sheets and the leaf $L'$ has a finite holonomy group of order $\frac{\left|\Hol(L,y)\right|}{k}$. 
\label{thm1}
\end{otherthm}

\subsection{Transverse disks fields}\label{s.disksfield}

Let $\cF$ be a foliation of a compact manifold $M$. We denote by $N_\cF$ its normal bundle: 
in an abstract way one can see it as the quotient of the tangent bundle $TM$ by the tangent bundle of 
$\cF$.  More concretely, if $M$ is endowed with a Riemannian metric, one could  realize the normal bundle as the 
orthogonal bundle $T\cF^{\perp}$ to the foliation. However, we deal here with non-smooth foliations so that 
the orthogonal bundle is not smooth.  
For this reason, we denote by $N_\cF$ a smooth subbundle of $TM$ 
transverse to the tangent bundle $T\cF$ of the foliation and close to the orthogonal bundle. 

For any $\rho>0$ we denote by $N_{\cF,\rho}$ the $\rho$-neighborhood of the zero section of 
$N_\cF$, that is, the set of vectors in $N_\cF$ with norm less or equal to $\rho$. 

For $\rho>0$ small enough, the exponential $\theta$ (associated to the Riemannian metric) 
defines a submersion from $N_{\cF,\rho}$ to $M$ and the restriction to each disk 
$$N_{\cF,\rho}(x)=\{u\in N_{\cF}(x), \|u\|\leq \rho\}$$ 
is an embedding which is transverse to $\cF$.  We denote   $$D_{x,\rho}=\theta(N_{\cF,\rho}(x)).$$
The family $\cD= \{D_{x,\rho}\}$ is a family of disks centered at the reference point $x$, 
transverse to $\cF$ and varying continuously with $x$ in the $C^1$ topology. 
Such a family is called a \emph{transverse disks field} . Using the compactness of the manifold, 
 the continuity of the family of disks, and the tranversality, one deduces (see for instance \cite[p.64, section 2.2.a]{Bo} 
where such transverse disks fields are defined and  used in the same way): 

\begin{lemma}\label{l.diskfield} There is $\rho>0$ so that
\begin{itemize}
\item for every $x~\neq~y$ in the same leaf $L$ of $\cF$ and with $d_\cF(x,y)~<~\rho$ we have \newline$D_{x,\rho}\cap D_{y,\rho}~=~\emptyset$;
\item for every $x, z~\in~M$, the intersection of $D_{x,\rho}\cap B_\cF(z,\rho)$ contains at most $1$ point, where  $B_\cF(z,\rho)$ denotes the ball of radius $\rho$ centered at $z$ inside the leaf of $z$.
\end{itemize}
\end{lemma}

\begin{defn} Let $L$ be a leaf of $\cF$ and $\gamma\colon[0,1]~\to~L$ a path on the leaf.  Denote $x~=~\gamma(0)$.  A path $\sigma$ is a projection of $\gamma$ along $\cD$ on a leaf of $\cF$ if 
\begin{itemize}
\item $\sigma $ is contained in a leaf of $\cF$ and
\item $\sigma(t)~\in~D_{\gamma(t),\rho}$ for every $t$. 
\end{itemize}
\end{defn}

\begin{lemma}\label{l.projection} Given a path $\gamma\colon[0,1]~\to~L$ inside a leaf $L$ of $\cF$ and $x~=~\gamma(0)$.  Then for every $y~\in~D_{x,\rho}$ there is at most one projection $\gamma_{y}$ of $\gamma$ on the leaf $L_y$, starting at $y$.  More precisely, the projection is unique and well defined until the distance  $d(\gamma_y(t),\gamma(t))$ becomes larger than $\rho$. 
\end{lemma}

\section{Holonomy covers}\label{s:covers}
From now on, we suppose that  $\cF$ is a uniformly compact foliation of a compact manifold $M$. 
As most proofs are quite intuitive for the case of trivial holonomy, our main tool are so-called \emph{holonomy covers} which eliminate (locally) the holonomy and whose existence is due to the Reeb Stability 
Theorem~\ref{thm1}. In these covers we have to define a suitable well-defined distance between lifted leaves which is equivalent to the canonical Hausdorff distance. 

%%%%%%%%%%%%%%%%%%%%%%%%%%%%%%%%%%%%%%%%%%%%%%%
\subsection{Holonomy covers}\label{sec:covers}
%%%%%%%%%%%%%%%%%%%%%%%%%%%%%%%%%%%%%%%%%%%%%%%%%

For every leaf $L$ of a foliation $\cF$ the holonomy covering map of $\cF$ is the covering map associated to the kernel of the holonomy group of $L$. 
If $\cF$ is uniformly compact, then the holonomy covering is a finite covering. Notice that this covering map extend to any tubular neighborhood of $L$ in $M$.
More precisely, we define the following: 
\begin{defn} A closed connected saturated set $V$ \emph{admits a holonomy covering} if there is a finite cover $p\colon\tilde V~\to~V$ such that the lift $\tilde \cF$ of $\cF$ on $\tilde V$ has trivial holonomy and if, for every leaf $L~\subset~V$, it induces the holonomy covering by restriction to every connected component of $p^{-1}(L)$. 
\end{defn}

\begin{defn}
We say that a finite cover $\cV=\{V_i,p_i\}$ of $M$ is a \emph{holonomy cover} for the foliation $\cF$  if
\begin{itemize}
\item every $V_i$ is a compact saturated set,
\item the interiors of the sets $V_i$ cover $M$, i.e. $M=\bigcup_i \interior(V_i)$, and
\item each $V_i$ admits a holonomy covering $p_i\colon\tilde V_i~\to~V_i$.
\end{itemize}
\end{defn}
The existence of a holonomy cover is given by the next Lemma~\ref{l.covering} which is a direct consequence of Theorem~\ref{thm1}: 
\begin{lemma}\label{l.covering} Let $\cF$ be a uniformly compact foliation and $L$ a leaf of $\cF$.  Then there exist a compact saturated neighborhood $V$ of $L$ and a cover $p\colon\tilde V~\to~V$ such that:
\begin{itemize}
\item the restriction of $p$ to $p^{-1}(L)$ is the holonomy covering map of $L$;
\item for every leaf $L'$ in $V$ and any connected component $\tilde L'$ of $p^{-1}(L')$, the restriction of $p$ to $\tilde L'$ is conjugate to the holonomy covering of $L'$. 
\end{itemize}
\end{lemma}

\begin{rem} 
\begin{itemize}  
\item Let $\tilde \cF$ be the lift of $\cF$ on the cover $\tilde V$ of $V$ in Lemma~\ref{l.covering}.  Then the foliation $\tilde \cF$ is uniformly compact and with trivial holonomy. 
\item Any finite cover $q\colon \hat V~\to~V$ such that the lift of $\cF$ on $\hat V$ has trivial holonomy,   is a finite cover of $\tilde V$, i.e. there is a finite cover $r\colon\hat V~\to~\tilde V$ such that $q=p\circ r$.
\end{itemize}
\end{rem}

\subsection{The metric and the holonomy covers}

We assume now that $M$ is endowed with a Riemannian metric,  $\cF$ is a uniformly compact foliation on $M$ and $\cV~=~\{V_i, p_i\}$ is a holonomy cover of $\cF$. 
For each $V_i$ we denote by $p_i\colon\tilde V_i~\to~V_i$ its holonomy covering.  We endow the interior of each $\tilde V_i$ with the Riemannian metric obtained by lifting the metric on $M$ by $p_i$. Therefore the projection $p_i$ is a local isometry if restricted to the interior of each $\tilde V_i$.

\begin{lemma}\label{l.diameter} Let $\tilde \cF$ denote the lift of the foliation $\cF$ on $\tilde V_i$. Then  the diameter (and the volume) of the leaves of $\tilde \cF$ for the lifted Riemannian metric is uniformly bounded independently from the holonomy cover.
\end{lemma}
\begin{proof} Let $\tilde L~\subset~\tilde V_i$ be a leaf of $\tilde \cF$.  Then $p_i(\tilde L)$ is a leaf $L$ of $\cF$. Furthermore, the projection $p_i\colon\tilde L~\to~L$ is a finite cover with number of sheets uniformly bounded by the maximal rank of the holonomy groups.  As the projection is a local isometry, the diameter of $\tilde L$ is bounded by the one of $L$ multiplied by the number of sheets.  As $\cF$ is uniformly compact, the diameter and volume of its leaves are uniformly bounded, concluding the proof. 
\end{proof}

\begin{lemma}\label{l.Lebesgue}  Given any holonomy cover $\cV=\{V_i, p_i\}$ of a uniformly compact foliation $\cF$, there is $\delta=\delta(\cV)~>~0$  so that for every leaf $L\in\cF$  there is $i$ so that:
\begin{itemize}
\item for every $x~\in~L$ the ball $B(x,\delta)$ is contained in the interior of $V_i$ and the projection $p_i\colon\tilde V_i~\to~V_i$ induces an isometry in the restriction to every connected component of $p_i^{-1}(B(x,\delta))$, 
\item in particular, the Hausdorff ball $B_H(L,\delta)$ is contained in the interior of $V_i$.
\end{itemize}
\end{lemma}

\begin{proof} The interiors of the $V_i$'s induce an open cover of the leaf space $M/\cF$, endowed with the Hausdorff distance. Therefore, there is a Lebesgue number $\delta_0$ of this cover.  
In other words, for every leaf $L\in\cF$ there is $i$ such that the Hausdorff ball $B_H(L,\delta_0)$ is contained in the interior of $V_i$.  One deduces from Corollary~\ref{lemma_neighborhoods} the existence of $\delta_1$ for which the union of the balls of radius $\delta_1$ centered at points $x \in L$ is contained in $B_H(L,\delta_0)$.  By shrinking $\delta_1$ if necessary one gets with $\delta:=\delta_1$ the announced isometry condition. 
\end{proof}

%%%%%%%%%%%%%%%%%%%%%%%%%%%%%%%%%%%%%%%%%%%%%%%%%%%%%%%%%%%%%%
\subsection{The holonomy covers and the Hausdorff metric}\label{Hausdorff metric}
%%%%%%%%%%%%%%%%%%%%%%%%%%%%%%%%%%%%%%%%%%%%%%%%%%%%%%%%%%%%%%%

\begin{rem} Let $\cV=\{V_i,p_i\}$ be a holonomy cover.
We still denote by $d_H$ the Hausdorff distance in each  $\tilde V_i$. 

For any leaves $L_1,L_2~\subset~V_i$  and any lifts $\tilde L_1,\tilde L_2$  in $\tilde V_i$ of  $L_1,L_2$  it holds that 
$$d_H(\tilde L_1,\tilde L_2) \geq d_H(L_1,L_2).$$
Therefore  for every leaf $L $ and $i$ with $B_H(L,\delta)~\subset~\interior(V_i)$, where $\delta=\delta(\cV)$ is given by Lemma~\ref{l.Lebesgue},  and any lift $\tilde L$ of $L$ on $\tilde V_i$, it follows that $B_H(\tilde L,\delta)~\subset~\interior(\tilde{V}_i).$ 
\end{rem}

\begin{lemma}
Let  $\cV=\left\{V_i,p_i\right\}$ be a holonomy cover and $\delta=\delta(\cV)~>~0$ 
be given by Lemma~\ref{l.Lebesgue}. 
For every $\delta_0~>~0$  there exists $\delta_1~>~0$, such that for any leaf $L$ and $i$ with $B_H(L,\delta)~\subset~V_i$, one has: 
\begin{itemize}
\item for every lift $\tilde L$ of $L$ on $V_i$ and every $\tilde x~\in~\tilde L$ one has 
$$\tilde \cF(B_{\delta_1}(\tilde x)) \subset B_H(\tilde L,\delta_0),$$ 
where $\tilde \cF$ denotes the lift of $\cF$ on $\tilde V_i$; 
\item 
 for every lift  $\tilde L~\subset~p^{-1}_i(L)$  and for every leaf $L_2~\subset~B_H(L,\delta_1)$ there exists a lift  $\tilde L_2~\subset~p^{-1}_i (L_2)$ such that $\tilde L_2~\subset~B_H(\tilde L,\delta_0)$. 
\end{itemize}

\label{lemma_distances}
\end{lemma}

According to Corollary~\ref{lemma_neighborhoods} for every uniformly compact foliation $\mathcal{F}$, for every $x~\in~M$ and every $\alpha~>~0$ there exists $\eta~>~0$ such that $\mathcal{F}(B_{\eta}(x))~\subset~B_H(L(x),\alpha)$ where $L(x)$ denotes the leaf through $x$. 

The first item is the same statement, on the $\tilde V_i$. The unique difficulty here are \emph{boundary effects}.  This is the reason why we require $B_H(L,\delta)~\subset~V_i$, so that $\tilde L$ remains at distance $\delta$ from the boundary of $\tilde V_i$. 

\begin{proof}[Idea of the proof] 
We fix a transverse disks field $\cD=\{D_{x,\rho}\}$  as in Section~\ref{s.disksfield}.
 
We want to project the leaf $\tilde L$  along $\cD$ onto the leaf through a point $\tilde y\in B_{\delta_1}(\tilde x)$, and we want to show that, for $\delta_1$ small enough, the distance of the projection remains smaller than $\delta_0$. 

We notice that the fact that $\tilde \cF$ is without holonomy can be expressed as follows: if one projects a path starting at $\tilde x$ onto the leaf through $\tilde y$, according to Lemma~\ref{l.projection}, then the projection of the end point does not depend on the path, (but just on this end point). Therefore, it is enough to consider paths whose length is bounded by the diameter of $\tilde L$.  According to Lemma~\ref{l.diameter}, this diameter is uniformly bounded, so that it is enough to consider paths with a priori bounded length. 

One covers $M$ by a finite atlas of foliated charts of $\cF$ so that each chart is contained in a ball of radius smaller than $\delta$. Therefore each chart meeting $L$ can be lifted isometrically on $\tilde V_i$. 

As the paths we consider have a uniformly bounded length, they cut a uniformly bounded number of charts. Let $K$ denote this bound. 

Now the announced bound follows from applying $K$ times the uniform continuity of the plaques in each chart: for any $\mu~>~0$ there is $\varepsilon~>~0$ so that if two plaques in a chart have points at distance less than $\varepsilon$, then the Hausdorff distance between the plaques is less than $\mu$. 
This proves the first item, choosing $\delta_1$ small enough so that the projection remains at distance less than $\inf(\delta,\delta_0)$.

The second item is a direct consequence: choose a point $x~\in~L$ and $y~\in~L_2$ so that $d(x,y)~<~\delta_1~<~\delta$. Let $\tilde x$ be a lift of $x$ on $\tilde L$.  There is a lift $\tilde y$ of $y$ for which $d(\tilde x,\tilde y)~<~\delta_1$.  The announced leaf $\tilde L_2$ is the leaf through $\tilde y$.

\end{proof}

The next lemma asserts that the Hausdorff distance of the lifts of leaves does not depend on the open neighborhood $V_i$ such that the leaves remain far from the boundary. 

\begin{lemma}[Local isometry] Let  $\cV=\left\{V_i,p_i\right\}$  be a holonomy cover and  
$\delta=\delta(\cV)~>~0$ given by Lemma~\ref{l.Lebesgue}.

Then there exists $\delta_0~>~0$ such that for every point  $x~\in~M$ and $i,j$ such that $B_H(L(x),\delta)~\subset~U_i~\cap~U_j$, for all $x_1~\in~p_i^{-1}(x), x_2~\in~p_j^{-1}(x)$ there exists an isometry 
$$I: B_H(\tilde L(x_1),\delta_0) \rightarrow B_H(\tilde L(x_2),\delta_0)$$ 

such that $p_j \circ I = \id_M \circ p_i$. Furthermore, $I(x_1)=x_2$ and $I$ induces a conjugacy between the restrictions of the lifted foliation $\tilde \cF$ to $B_H(\tilde L(x_1),\delta_0)$ and $B_H(\tilde L(x_2),\delta_0)$.
\label{lemma_independence_cover}
\end{lemma}
\begin{proof}[Idea of the proof] First notice that the restriction of $p_i$ and $p_j$ to $\tilde L(x_1)$ and $\tilde L(x_2)$ are copies of the holonomy cover of $L(x)$. Furthermore, they are local isometries following Lemma~\ref{l.Lebesgue} for $\delta_0 < \delta(\cV)$. One deduces that there is an isometry $I_L\colon \tilde L(x_1)~\to~\tilde L(x_2)$ with $I_L(x_1)=x_2$. 

We consider the transverse disks field $\cD$ of radius $\rho$ given by Lemma~\ref{l.diskfield} and smaller than $\delta$.  As a consequence, for every $y~\in~L$ and any lift $\tilde y$ of $y$ by $p_i$ or $p_j$, there is a well defined lift of $D_{y,\rho}$ centered at $\tilde y$, and this lift is an isometry. Let us denote by $D_{\tilde y,\rho}$ this lift. 

As $\tilde \cF$ is without holonomy, one can prove that
\begin{claim}
for $\rho$ small enough, for every leaf $L$ and $i$ so that $B_H(L,\delta)~\subset~V_i$, for every leaf $\tilde L~\subset~p_i^{-1}(L)$ and for every 
$\tilde x~\neq~\tilde y\in \tilde L$ one has $D_{\tilde x,\rho}\cap D_{\tilde y,\rho}=\emptyset$.

In other words, $\bigcup_{\tilde x\in\tilde L} D_{\tilde x,\rho}$ is a tubular neighborhood of $\tilde L$.
\end{claim}

Now, the announced isometry $I$ is defined as 

$$I|_{D_{\tilde y,\rho}}= \left(p_j|_{D\left(I_L( p_i ( \tilde y)),\rho\right)}\right)^{-1}\circ p_i.$$
\end{proof}

As  a direct consequence of Lemmata~\ref{lemma_distances} and \ref{lemma_independence_cover} one gets: 

\begin{corol}\label{corol_generality} Let $\cV=\{V_i\}$ be a holonomy cover. 
Consider a leaf $L$ and $i,j$ such that $B_H(L, \delta)~\subset~V_i\cap V_j$. Assume that there are $L_1\subset V_i$, and lifts $\tilde L^i~\in~p^{-1}_i(L)$ and $\tilde L_1^i~\in~p^{-1}_i(L_1)$  such that $$d_H(\tilde  L^i,\tilde L^i_1)< \delta_0.$$

Then  for every lift $\tilde L^j~\subset~p_j^{-1}(L)$ there exists $\tilde L_1^j~\in~p^{-1}_j(L_1)$ such that $$d_H(\tilde L^j,\tilde L_1^j)~=~d_H(\tilde L^i,\tilde L_1^i)~\leq~\delta_0.$$
\end{corol}

\subsection{A modified Hausdorff distance}
We consider a holonomy cover $\cV$ and $\delta=\delta(\cV)$ given by Lemma~\ref{l.Lebesgue}. Let $\delta_0~>~0$  be the constant associated to $\frac\delta2$ by Lemma \ref{lemma_independence_cover}, and $\delta_1~>~0$ associated to $\frac{\delta_0}2$ by  Lemma~\ref{lemma_distances}. We can assume that $\delta_1~<~\frac{1}{4}\delta$.

To every pair $L_1,L_2$ of  leaves we associate a number $\De_H(L_1,L_2)$ as follows: 
$$\Delta_H(L_1,L_2):=\begin{cases}  &\frac{1}{2}\delta_0 ,\quad \text{if}\; d_H(L_1,L_2) \geq \delta_1, \mbox{and}\\
 &\min \left\{\frac{1}{2}\delta_0,d\right\} ,\quad \text{if}\; d_H(L_1,L_2) \leq \delta_1,\\
 &\mbox{ where  $d$ is the lower bound of the }d_H(\tilde L_1,\tilde L_2)\\
 &\mbox{ for all } \tilde L_1 \in p^{-1}_i(L_1),\tilde L_2 \in p^{-1}_i(L_2)\;\mbox{ such that  }  B_H(L_1,\delta)\subset V_i.\end{cases}.$$

We denote a ball with respect to  $\De_H$ by $$B_{\Delta_H}(L,\delta):=\left\{L_2 \subset M\;\big|\; \Delta_H(L,L_2) < \delta\right\}.$$

As a corollary of Lemma \ref{lemma_distances} we prove:
\begin{corol}\label{corol_distance} $\Delta_H$ is a distance, and it is topologically equivalent to $d_H$, i.e. for any leaf $L$ and any $\alpha~>~0$ there exist $\beta_1,\beta_2~>~0$  such that 
\begin{align}\label{equation_equivalence}
B_H(L,\beta_1) \subset B_{\Delta_H}(L, \alpha) \mbox{ and } B_H(L,\beta_2) \supset B_{\Delta_H}(L, \alpha).
\end{align}
\end{corol}
\begin{proof}
The positive definiteness is directly inherited by $d_H$. The symmetry comes from the fact that the lower bound of $d_H(\tilde L_1,\tilde L_2)$ for any $\tilde L_1~\in~p^{-1}_i(L_1)$,\newline$\tilde L_2~\in~p^{-1}_i(L_2)$ does not depend on the $i$ for which $B_H(L_1,\frac\delta 2)~\subset~V_i$, according to Lemma~\ref{lemma_independence_cover}.  As we consider $i$ for which $B_H(L_1,\delta)~\subset~V_i$ and $d_H(L_1,L_2)~<~\delta_1~<~\delta/2$, one also has $B_H(L_2,\delta/2)~\subset~V_i$  leading to the symmetry. \\
For the triangle inequality we consider three leaves $L_1,L_2,L_3$ and we need to prove $$\De_H(L_1,L_3)~<~\De_H(L_1,L_2)~+~\De_H(L_1,L_3).$$  
This is easy if $\De_H(L_1,L_2)~+~\De_H(L_1,L_3)~\geq~\frac12 \delta_0$. We choose $i$ for which $B_H(L_2,\delta)~\subset~V_i$ and we consider lifts $\tilde L_i$ so that $d_H(\tilde L_1,\tilde L_2)~=~\De_H(L_1,L_2)$ and $d_H(\tilde L_3,\tilde L_2)~=~\De_H(L_3,L_2)$.  Then $d_H(L_1,L_3)~<~\De_H(L_1,L_2)~+~\De_H(L_1,L_3)$ (triangular inequality for the Hausdorff distance) and $\De_H(L_1,L_3)~\leq~d_H(L_1,L_3)$.

\vskip 2mm 

We now show that $\Delta_H$ is a topologically equivalent distance to $d_H$. Consider a leaf $L$ and $\alpha~>~0$.

Notice that $d_H(\tilde L_1,\tilde L_2)~\geq~d_H(L_1,L_2)$ so that $\De_H~\geq~\inf\{\frac 12\delta_0, d_H\}$. Therefore, 
$$B_{\De_H}(L, \inf\{\frac 14\delta_0,\alpha\})~\subset~B_H(L,\alpha).$$

The converse inclusion follows from  Lemma~\ref{lemma_distances}: there exists $\beta_2~>~0$ such that any $L_2~\subset~B_H(L,\beta_2)$ admits a lift so that $d_H(\tilde L,\tilde L_2)~<~\alpha$ proving 
 $B_H(L,\beta_2)~\subset~B_{\De_H}(L,\alpha)$.

\end{proof}

\section{Dynamical coherence}\label{sec:coherence}

The aim of this section is to prove  Theorem~\ref{theorem_dc}, i.e. that
partially hyperbolic diffeomorphisms with an invariant uniformly compact
center foliation are dynamically coherent. In other words, we have to prove firstly
that, if $x,y$ are two points in the same stable leaf, that is $y~\in~W^s(x)$, then 
the center leaf $W^c(y)$ through $y$ is contained in the union $\bigcup_{z \in W^c(x)}W^s(z)$ of stable leaves through
 the center leaf of $x$. Further, we have to show that for any $x\in M$ the set $\bigcup_{z\in W^c(x)}W^s(z)$ forms a leaf of the center stable foliation. 

The idea of the proof is very simple: iterating $x$ and $y$ by $f$, the distance
decreases exponentially. Consequently, the Hausdorff distance 
 $d_H(W^c(f^n(x)),W^c(f^n(y))$ tends to $0$ for $n~\rightarrow~\infty$.  
On the other hand, if $W^c(y)$ has a point $w$ outside $\bigcup_{z\in W^c(x)}W^s(z)$,
its unstable manifold $W^u(w)$ cuts $\bigcup_{z\in W^c(x)}W^s(z)$ 
at some point $w$ which is, in some sense, the projection of $z$ on 
$\bigcup_{z\in W^c(x)} W^s(z)$. Keep in mind that this projection might be not unique!
If the distance $d_H(W^c(x),W^c(y))$ is
very small, the projection distance $d(z,w)$ will be small. 
At the same time, the distance $d(f^n(z),f^n(w))$ increases exponentially.  
 
 If the center foliation $\cF^c$ is without holonomy, the projection $w$ of $z$ is uniquely defined, 
 so that we get a contradiction. However, we do not assume $\cF^c$ without holonomy
 so that this projection distance depends on choices.
 Therefore, to solve this problem,we lift the argument on a holonomy cover.   
 
 The argument cannot be so simple, as shows the example in Proposition~\ref{p.example}: the 
 fact that  $d_H(W^c(f^n(x)),W^c(f^n(y))$ tends to $0$ does not imply that 
 $W^c(y)$ is contained in $\bigcup_{z \in W^c(x)}W^s(z)$. The difficulties come from the fact that 
 each $W^c(f^n(x))$ and $W^c(f^n(y))$ may admit many lifts on a holonomy cover. The key point consists in 
 choosing, for every lift of $W^c(f^n(x))$ the appropriate lift of $W^c(f^n(y))$.

\subsection{Local product structure}
We want to prove that the stable and the center foliations are jointly integrable. 
A first step to that is to show that the three foliations -  the stable, 
center and unstable - form a kind of local product structure: 
the local unstable manifold of a point intersects in exactly $1$ point the union 
of the local stable manifold through the local center manifold of a nearby point. 
 That is the aim of the next two lemmas: 
 
\begin{itemize}
\item  Lemma~\ref{dc_lemma1a} asserts the existence of the intersection point.  
That is a very general argument, which holds for any triple of transverse foliations.

\item Lemma~\ref{dc_lemma1} proves the uniqueness of the intersection point. 
This would be wrong for general $\cC^{1,0+}$-foliations: 
it holds for dynamical reasons using the fact that we deal with the stable, 
center and unstable foliations of a partially hyperbolic diffeomorphism.
\end{itemize}

\subsubsection{Triple of pairwise transverse foliations}
Firstly, we establish the existence of an intersection point:
\begin{lemma}
Let $\mathcal{F}_1, \mathcal{F}_2, \mathcal{F}_3$ be three continuous foliations 
with smooth leaves of a compact smooth manifold $M$, tangent to continuous 
distributions $E_1, E_2, E_3$ such that $TM=E_1 \oplus E_2 \oplus E_3$. 
Then there exist  $\delta_0~>~0$ and $C~>~0$ such that for all $\delta~<~\delta_0$ 
and $x,y$ with $d(x,y)~<~C\delta$ the intersection 
$F_{1,\delta}(x) \cap \bigcup_{z \in F_{2,\delta}(y)}(F_{3,\delta}(z))$ is non-empty. 
\label{dc_lemma1a}
\end{lemma}
\begin{proof}[Proof of Lemma~\ref{dc_lemma1a}]
Let $d_1,d_2,d_3$ denote the dimensions of the foliations $\cF_1,\cF_2,\cF_3$ in the lemma.
Then $d_1+d_2+d_3~=~d~=~\dim M$.  Lemma~\ref{dc_lemma1a} follows 
from the  Lemmas~\ref{l.charts} and ~\ref{l.quasi-constant} below.

Lemma~\ref{l.charts} asserts that, in sufficiently small local coordinates, 
the bundles $E_1,E_2$ and $E_3$ can be assumed to be arbitrarily close 
to the constant bundles 
$\RR^{d_1}\times\{0\}^{d_2+d_3}$, $\{0\}^{d_1}\times \RR^{d_2}\times\{0\}^{d_3}$ 
and $\{0\}^{d_1+d_2}\times \RR^{d_3}$.

\begin{lemma}\label{l.charts} Under the hypothesis of Lemma~\ref{dc_lemma1a}, 
for every $c~>~0$ there is a finite family $(\psi_n)_n$ of smooth embeddings 
$\psi_n\colon[-4,4]^d~\to~M$ such that the
$\psi_n (]-1/2,1/2[^d)$ form a finite open cover of $M$ and, 
for every $n$,  there are vector fields $e_1,\dots,e_d $ on $[-4,4]^d$ 
satisfying:
\begin{itemize} 
\item the plane fields $(D\psi_n)^{-1}(E_1)$, $(D\psi_n)^{-1}(E_2)$, 
and $(D\psi_n)^{-1}(E_3)$ are generated  by 
$(e_1,\dots,e_{d_1})$, $(e_{d_1+1},\dots,e_{d_1+d_2})$ 
and $(e_{d_1+d_2+1},\dots,e_{d})$, respectively;
\item let $(e_{i,1}(x)\dots e_{i,d}(x))$ 
denote the coordinates of $e_i(x)$ in the canonical basis of 
$\RR^d$. Then 
\begin{itemize}
\item $e_{i,i}=1$ and $|e_{i,j}|~<~c$ for $i~\neq~j$;
\item $e_{i,j}=0$ if $i~\neq~j$  and  
$(i,j)\in\{1,\dots d_1\}^2$ or $(i,j)\in\{d_1+1,\dots d_1+d_2\}^2$ 
or $(i,j)\in\{d_1+d_2+1,\dots d\}^2$.
\end{itemize}

\end{itemize}
\end{lemma}
The proof of Lemma~\ref{l.charts} is straightforward and just follows from the compactness of $M$, 
the continuity and the transversality of the subbundles $E_i$. 

Lemma~\ref{l.quasi-constant} below uses the local coordinates given by Lemma~\ref{l.charts} for proving 
the existence of the intersection point announced by  Lemma~\ref{dc_lemma1a}.

\begin{lemma}\label{l.quasi-constant} There is $C_d>0$ so that, for every $c$ satisfying  
$0<c~<~C_d$ one has the following property.

Let $\cF_1,\cF_2,\cF_3$  be three pairwise transverse $C^0$-foliations of 
$[-4,4]^d$ tangent to $C^0$-bundles $E_1$ $E_2$,$E_3$ of dimensions 
$d_1,d_2$ and $d_3$ with $d_1+d_2+d_3=d$ and satisfying the conclusions
of  Lemma~\ref{l.charts} for a constant $c$, that is: 

There are vector fields $e_1,\dots e_d $ on $[-4,4]^d$ satisfying: 
\begin{itemize} 
\item the plane fields $(D\psi_n)^{-1}(E_1)$, $(D\psi_n)^{-1}(E_2)$, 
and $(D\psi_n)^{-1}(E_3)$ are generated  by 
$(e_1,\dots,e_{d_1})$, $(e_{d_1+1},\dots,e_{d_1+d_2})$ and 
$(e_{d_1+d_2+1},\dots,e_{d})$, respectively;
\item let $(e_{i,1}(x)\dots e_{i,d}(x))$ 
denote the coordinates of $e_i(x)$ in the canonical basis of $\RR^d$. 
Then 

\begin{itemize}
\item $e_{i,i}=1$ and $|e_{i,j}|~<~c$ for $i~\neq~j$;
\item $e_{i,j}=0$ if $i~\neq~j$  and  $(i,j)~\in~\{1,\dots d_1\}^2$ or 
$(i,j)~\in~\{d_1+1,\dots d_1+d_2\}^2$ or $(i,j)~\in~\{d_1+d_2+1,\dots d\}^2$.
\end{itemize}
\end{itemize}

Then for any pair of points $x,y~\in~[-1,1]^d$, the leaf of $\cF_1$ through $x$ 
cuts the union of the leaves of $\cF_3$ through the points of $\cF_2$ 
through $y$: 
$$\cF_1(x)\cap\bigcup_{z\in\cF_2(y)} \cF_3(z)\neq\emptyset.$$

\end{lemma}

\begin{proof}[Proof of Lemma~\ref{l.quasi-constant}]One easily checks that, for every $c$ small enough, 
the leaves $\cF_i(x)$ of $\cF_1$, $\cF_2$ and $\cF_3$ through points 
$x~\in~[-3,3]^d$ are complete graphs over the cube 
$[-4,4]^{d_1}\times\{0\}^{d_2+d_3}$,$\{0\}^{d_1}\times[-4,4]^{d_2}\times\{0\}^{d_3}$, 
and  $\{0\}^{d_1+d_2}\times[-4,4]^{d_3}$, 
respectively. 
Furthermore, for $x~\in~[-1,1]^d$ the intersection of 
$\cF_2(x)$ with $[-3,3]^d$ is a complete graph over  
$\{0\}^{d_1}\times[-2,2]^{d_2}\times\{0\}^{d_3}$.

For every point $y=(y_1,\dots, y_d)~\in~[-1,1 ]^d$ one considers the union 
$F_{2,3}(y)$ of the $\cF_3$-leaves through $\cF_2(y)\cap[-3,3]^d$. 
These $\cF_3$-leaves are pairwise disjoint so that $F_{2,3}$ is a 
topological embedding of $[-3,3]^{d_2}\times [-4,4]^{d_3}$, 
in particular, it is a $(d_2+d_3)$-disk. Furthermore, as the constant $c$ tends 
to $0$, $F_{2,3}(y)$ tends uniformly to 
$\{(y_1,\dots,y_{d_1})\}\times [-3,3]^{d_2}\times [-4,4]^{d_3}$ for the 
$C^0$-topology. 

Then, the boundary $\partial F_{2,3}(y)$ is a $(d_2+d_3)-1$-sphere, 
and tends uniformly to  
$\{(y_1,\dots,y_{d_1})\}\times \partial\left([-3,3]^{d_2}\times [-4,4]^{d_3}\right)$. 
In particular, for $c$ small enough, $\partial F_{2,3}(y)$ is disjoint 
from $[-4,4]^{d_1}\times [-2,2]^{d_2+d_3}$ and is homotopic to \newline
$\{(y_1,\dots,y_{d_1})\}\times \partial\left([-3,3]^{d_2}\times [-4,4]^{d_3}\right)$ 
in $[-4,4]^d\setminus [-4,4]^{d_1}\times [-2,2]^{d_2+d_3}$. 

In particular, one gets that for every $x,y~\in~[-1,1]^d$,  $\partial F_{2,3}(y)$ 
is homotopic to \newline 
$\{(y_1,\dots,y_{d_1})\}\times \partial\left([-3,3]^{d_2}\times [-4,4]^{d_3}\right)$ 
in $[-4,4]^d\setminus [-4,4]^{d_1}\times \cF_1(x)$.

This implies that $\partial F_{2,3}(y)$ is not homotopic to a point 
in $[-3,3]^d\setminus \cF_1(x)$. As a consequence, we prove 
$F_{2,3}(y)\cap F_1(x)\neq \emptyset$, ending the proof of 
Lemma~\ref{l.quasi-constant}. 

\end{proof}
The proof of Lemma~\ref{l.quasi-constant} also finishes the proof of Lemma~\ref{dc_lemma1a}
\end{proof}

\subsubsection{Transverse foliations associated to a partially hyperbolic 
diffeomorphism}

The next lemma ensures that the intersection point is indeed unique.  
The argument is very similar to the one  presented in \cite[Lemma 7]{B01} 
and has been also explained to the second author in a personal communication with Brin. 

We consider now a partially hyperbolic diffeomorphism $f$ on a compact manifold $M$, with $f$-invariant 
center foliation $\cF^c$.  
We denote by $\lambda$ an upper bound for $\|Df\|$. According to \cite{G07} we can use an 
equivalent adapted metric such that there exists constants $0~<~\alpha~<~1~<~\beta$ such 
that the norm $\left\|Df|_{E^s}\right\|$ is less than $\alpha$ and 
$\left\|Df^{-1}|_{E^u}\right\|$ is less than $\beta^{-1}$. 

\begin{rem}\label{r.distance}

\begin{itemize}
\item As the center foliation $\cF^c$ is tangent to a continuous bundle, its leaves 
are \emph{uniformly $C^1$}.  Therefore, the distance (in the manifold) between points 
in the same foliated \emph{plaque} is equivalent to the distance in the center leaf. More 
precisely, we consider a finite foliated atlas $(U_i)_i$ of $\cF^c$. The plaques 
are the leaves of the restrictions of $\cF^c$ to the charts $U_i$. 
For every $\eta~>~0$  there is $\mu$ so that, if $x,y$ are points in the same leaf of 
$\cF^c$ such that $d(x,y)~<~\mu$ and $d^c(x,y)~>~(1+\eta)d(x,y)$ then there is $i$ so 
that  $x$ and $y$ belong to $U_i$ but are not in the same plaque.  As a consequence, 
there is a constant $\De^c~>~0$ so that for all pair of points $x,y \in U_i$, but not in the same plaque, we have $d^c(x,y)~>~\De^c$. The same holds for the stable and unstable foliations  $\cF^s$ and $\cF^u$. 
\item If $\cF^c$  is a compact foliation with trivial holonomy,
then two  points in the same leaf which are nearby in the ambient manifold  are in 
the same plaque.  Therefore, the item above can be reformulated as follows: for any 
$\eta~>~0$  there is $\mu~>~0$ so that, if $x,y$ are in the same leaf and if $d(x,y)~<~\mu$ 
then $d^c(x,y)~<~(1+\eta)d(x,y)$.  

\item The foliations $\cF^s,\cF^c,\cF^u$ are tangent to continuous  bundles 
$E^s,E^c,E^u$ which are in direct sum.  As a consequence for $\mu$ small enough, for 
every $x~\in~M$ the intersection $W^u_\mu(x)\cap W^c_\mu(x)$ contains at most $1$ 
point: $W^u_\mu(x)\cap W^c_\mu(x)=\{x\}$.
 \end{itemize}
\end{rem}

\begin{lemma}
Let $f:M~\rightarrow~M$ be a partially hyperbolic $C^1$-diffeomorphism with an $f$-
invariant center foliation $\mathcal{F}^c$, and let $C~>~0$ be the constant associated 
to $\cF^s,\cF^c,\cF^u$ by Lemma~\ref{dc_lemma1a}. Then 
\begin{itemize}
\item there exist a constant  $\mu_0~>~0$  and for any $x,y~\in~M$ with $d(x,y)~<~\mu_0$ the intersection 
$$W^u_{\mu_0}(y) \cap \left(\bigcup_{z \in W^c_{\mu_0}
(x)}W^s_{\mu_0}(z))\right)$$ contains at most one point;

\item furthermore, if $d(x,y)\leq~\frac{\mu_0}C$ then this intersection is a 
unique point, which belongs to 
$$W^u_{C\mu}(y) \cap \left(\bigcup_{z \in W^c_{C\mu}
(x)}W^s_{C\mu}(z))\right),$$ 
where $\mu=d(x,y)$. 
\end{itemize}
\label{dc_lemma1}
\end{lemma}

\begin{figure}
\includegraphics[width=0.5\textwidth]{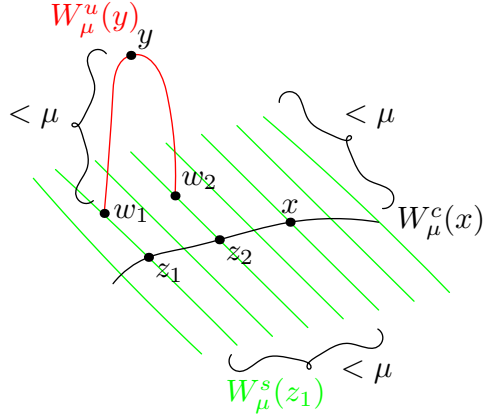}
\caption{Assume: for any $\mu > 0$ there are points $x,y$ with $d(x,y)<\mu$ such that $W^u_{\mu}(y)$ cuts $\bigcup_{z \in W^c_{\mu}(x)}W^s_{\mu}(z)$ in two distinct points $w_1$ and $w_2$.} 
\label{fig:intersection}
\end{figure}

\begin{proof} 

One argues by contradiction assuming that, for every $\mu~>~0$  there is $x,y$ at a distance $d(x,y) < \mu$ such 
that $W^u_{\mu}(y)$ cuts $\bigcup_{z \in W^c_{\mu}(x)}W^s_{\mu}(z)$ in two distinct 
points, see Figure~\ref{fig:intersection}. According to Remark~\ref{r.distance} there is  $\mu_*~>~0$ so that: 
\begin{itemize}
 \item if two points $p,q$ 
in the same center leaf are at distance less than $\mu_*$ in the manifold, but at 
distance more than $2d(p,q)$ in  the center leaf, then they are at distance larger 
than $8\lambda^2\mu_*$ in the center leaf where $\lambda~>~1$ is a upper bound for $\|
Df\|$, i.e. for any 
\begin{equation}\label{e.property1}
p,q \in M\,\mbox{with}\, d(p,q)<\mu_*\,\mbox{and}\,d^c(p,q)>2d(p,q)\; \Rightarrow\; d^c(p,q) > 8\lambda^2\mu_*;
\end{equation}
\item if two points $p,q$ 
in the same unstable leaf are at distance less than $\mu_*$ in the manifold, but at 
distance more than $2d(p,q)$ in  the unstable leaf, then they are at distance larger 
than $8\lambda^2\mu_*$ in the unstable leaf, i.e. for any
\begin{equation}\label{e.property2}
p,q \in M\,\mbox{with} d(p,q)<\mu_*\,\mbox{and}\,d^u(p,q)>2d(p,q)\; \Rightarrow\; d^u(p,q) > 8\lambda^2\mu_*;
\end{equation}
\item for every point $x\in M$  one has $W^u_{\mu_*}(x)\cap W^c_{\mu_*}
(x)=\{x\}$.
\end{itemize}

Consider now $\mu~<~\mu_*/8$. According to our assumption, there are $x=x(\mu),y=y(\mu)$ such that $W^u_{\mu}(y)$ cut 
$\bigcup_{z \in W^c_{\mu}(x)}W^s_{\mu}(z)$ in two points $w_1, w_2$ with $w_i~\in~W^s_\mu(z_i)$, 
$z_i~\in~W^c_\mu(x)$.  As $z_1,z_2~\in~W^c(x)$ one gets  
$d^c(z_1,z_2)~<~2\mu~<~\mu_*$.

To get uniqueness, let $0~<~\alpha~<~1~<~\beta$ such that vectors in the stable bundle are contracted by 
$\alpha$ and vectors in the unstable bundle are expanded by $\beta$. 

Consider the iterates  $f^n(x),f^n(y),f^n(w_1),f^n(w_2),f^n(z_1),f^n(z_2)$ for $n~>~0$: 
\begin{itemize}
\item The stable distance $d^s(z_i,w_i)$ is contracted by at least $\alpha$ at 
each iteration: that is $$d^s(f^n(w_i),f^n(z_i))~\leq~\alpha^n 
d^s(w_i,z_i)~<~\alpha^n\mu~<~\mu_*/8 \;\mbox{for}\;i=1,2;$$
\item The unstable distance  $d^u(w_1,w_2)$ is expanded by at least $\beta$ but 
at most  $\lambda$ at each iteration: that is $$\beta^n d^u(w_1,w_2)~\leq~d^u(f^n(w_1),f^n(w_2))~\leq~\lambda^n d^u(w_1,w_2);$$
\item The center distance $d^c(z_1,z_2)$ is at most expanded by $\lambda$ at 
each iteration: $d^c(f^n(z_1),f^n(z_2))~\leq~\lambda^n d^c(z_1,z_2)$.
\end{itemize}

Therefore, there is $n_\mu~>~0$ so that 
\begin{equation}\label{e.zmu}
d^c(f^{n_\mu}(z_1),f^{n_\mu}(z_2))+d^u(f^{n_\mu}(w_1),f^{n_\mu}(w_2))\in [\lambda^{-1}\mu_*, \mu_*].
\end{equation} 
Let us denote $z_i(\mu)=f^{n_\mu}(z_i)$ and $w_i(\mu)=f^{n_\mu}(w_i)$.  Clearly, $n_\mu$ tends to 
$\infty$ when $\mu~\to~0$.  One deduces that $d^s(z_i(\mu),w_i(\mu))$ tends to $0$ when 
$\mu$ tends to $0$. 

We consider a sequence $\mu_k$ tending to $0$ so that the points 
$z_1(\mu_k),z_2(\mu_k),w_1(\mu_k),w_2(\mu_k)$ converge to points 
$z_1^\infty,z_2^\infty,w_1^\infty,w_2^\infty$.  As $d^s(z_i(\mu_k),w_i(\mu_k))$ tends 
to $0$ one gets $z_i^\infty=w_i^\infty$ for $i=1,2$.

As $d^u(w_1(\mu_k), w_2(\mu_k))~<~\mu_*$ and $d^c(z_1(\mu_k), z_2(\mu_k))~<~\mu_*$, one 
deduces that $z_1^\infty$($= w_1^\infty$) 
and  $z_2^\infty$($= w_2^\infty$) belong to the same unstable leaf and to the same 
center leaf and one has 
$$ d^u(z_1^\infty,z_2^\infty)<\mu_* \mbox{ and  }  d^c(z_1^\infty,z_2^\infty)<\mu_*.$$

In other words, $z_2^\infty~\in~W^u_{\mu_*}(z_1^\infty)\cap W^c_{\mu_*}(z_1^\infty)$. The choice of $\mu_*$ 
implies that 
\begin{equation}\label{e.equal} z_1^\infty=z_2^\infty.
\end{equation}
  
Therefore we conclude the proof of uniqueness 
of the intersection point by proving:
\begin{claim} $z_1^\infty\neq z_2^\infty$.  More precisely 
$d(z_1^\infty,z_2^\infty)~\geq~\frac{\mu_*}{4\lambda}$.
\end{claim} 
\begin{proof}By property~(\ref{e.zmu}) for $z_i(\mu),w_i(\mu)$ one has 
$$d^c(z_1(\mu_n),z_2(\mu_n))+d^u(w_1(\mu_n),w_2(\mu_n))~\in~\left[\frac{\mu_*}\lambda,\mu_*\right].$$ 
By properties~(\ref{e.property1}) and (\ref{e.property2}) of $\mu_*$ one  has 
\begin{align*}
&d(z_1(\mu_n),z_2(\mu_n))\geq \frac12 d^c(z_1(\mu_n),z_2(\mu_n))\\
&d(w_1(\mu_n),w_2(\mu_n)) \geq \frac12 d^u(w_1(\mu_n),w_2(\mu_n))\;\mbox{so that} \\
&d(z_1(\mu_n),z_2(\mu_n))+d(w_1(\mu_n),w_2(\mu_n))\in\left[\frac{\mu_*}{2\lambda},\mu_*\right].
\end{align*}
As a consequence one gets the announced inequality: 
$2 d(z_1^\infty,z_2^\infty)~\geq~\frac{\mu_*}{2\lambda}$.
\end{proof}
The claim contradicts (\ref{e.equal}) so that our assumption that there exist such points $x(\mu)$, $y(\mu)$
for $\mu$ arbitrarily small is wrong. This proves the existence of $\mu_0>0$ 
satisfying the first claim of the lemma.

The second claim of Lemma~\ref{dc_lemma1} is now a straightforward consequence of the 
first claim (uniqueness) and Lemma~\ref{dc_lemma1a} (existence).   
\end{proof}

\begin{rem} The same argument given by Brin allows him to get a stronger result: he 
shows that, for $\mu$ small enough, $\bigcup_{z \in W^c_{\mu}(x)}W^s_{\mu}(z)$ is a 
submanifold tangent to $E^c\oplus E^s$.  The idea is that every disc tangent to a cone 
field around $E^u$ with an arbitrarily large angle may cut $\bigcup_{z \in W^c_{\mu}
(x)}W^s_{\mu}(z)$ in at most $1$ point; the proof of this fact is exactly our 
argument. We did not found any published reference for this interesting and useful 
result. 
\end{rem}

%%%%%%%%%%%%%%%%%%%%%%%%%%%%%%%%%%%%%%%%%%%%%%%%%%%%%%%%%%%%%%
\subsection{Unstable projection  on  centerstable manifolds} 
%%%%%%%%%%%%%%%%%%%%%%%%%%%%%%%%%%%%%%%%%%%%%%%%%%%%%%%%%%%%%%
We define in this section the unstable projection of points onto the center stable manifold. We start with the case of a center foliation with trivial holonomy.
%%%%%%%%%%%%%%%%%%%%%%%%%%%%%%%%%%%%%%%%%%%%%%%%%%%%%%%%%%
\subsubsection{Center foliation with trivial holonomy}
%%%%%%%%%%%%%%%%%%%%%%%%%%%%%%%%%%%%%%%%%%%%%%%%%%%%%%%%%%%

\begin{lemma}
Let $f: M \rightarrow M$ be a partially hyperbolic $C^1$-diffeomorphism with an $f$-
invariant compact center foliation with trivial holonomy. Let $C>0$ be given by Lemma~\ref{dc_lemma1}. 

Then there exists $\mu > 0$ such that for any $x \in M$  and any leaf $L^c \in \cF^c$ with $d(x,L^c) < \mu/C$ the
intersection $W^u_{\mu}(x)~\cap~\left(\bigcup_{z \in L^c}W^s_{\mu}(z)\right)$ consists 
of a unique point. 
\label{dc_lemma2}
\end{lemma}

\begin{proof}

The proof follows directly from 
Lemma~\ref{dc_lemma1}. Let us prove the uniqueness.  Consider $\mu_0~>~0$  given by Lemma~\ref{dc_lemma1}.   
Consider $\mu~<~\frac{\mu_0}{4C}$, and assume that the intersection $W^u_{\mu}
(x)\cap\left(\bigcup_{z \in L^c}W^s_{\mu}(z)\right)$ contains two points $w_i~\in~W^s_\mu(z_i)$, 
$z_i~\in~L^c$, $i=1,2$.  Then the distance $d(z_1,z_2)$ is bounded by 
$4\mu$.  Assume now that $\mu$ is small enough, according to Remark~\ref{r.distance} 
so that $d(z_1,z_2)~<~4\mu$ implies that $d^c(z_1,z_2)~<~5\mu$.  Therefore $w_1,w_2$ 
belongs to $W^u_{\mu}(x)\cap\left(\bigcup_{z \in W^c_{5\mu}}W^s_{\mu}(z)\right)$. Thus, if $5\mu~<~\mu_0$, one deduces $w_1=w_2$ which concludes the proof. 
\end{proof}

\begin{rem}\label{dc_rem}
One easily checks that 
\begin{itemize}
\item the map $(x,L^c)~\mapsto~w= W^u_{\mu}(x)~\cap~\bigcup_{z\in L^c} W^s_{\mu}
(z)$ is continuous, for $d(x,L^c)~<~\frac{\mu}C$ and $\mu$ given by Lemma~\ref{dc_lemma2} above. 
\item the unstable distance $d^u(x,w)$ is bounded by $Cd(x,L^c)$.  
\end{itemize}
\end{rem}

\subsubsection{Unstable projection on the stable manifold of a center leaf in the 
holonomy cover}

We assume now that $f$ is a partially hyperbolic diffeomorphism with a uniformly 
compact center foliation $\cF^c$.  We denote by $\cV=\{V_i\}$ a holonomy cover of 
$\cF^c$. 

\begin{lemma}Let $f:M~\rightarrow~M$ be a partially hyperbolic $C^1$-diffeomorphism 
with an $f$-invariant uniformly compact center foliation, and let $\cV=\{V_i\}$ be a 
holonomy cover and $\delta$ associated to $\cV$ by Lemma~\ref{l.Lebesgue}.

Then there exists $\mu_1~>~0$ such that
\begin{itemize}
\item  for every leaf $L^c$ of $\cF^c$ and all $i$ so that the $\delta$-neighborhood 
of $L^c$ is contained in $V_i$;
\item for every lift $\tilde L^c$ of $L^c$ on $\tilde V_i$
\item for every $\tilde x\in \tilde V_i$ with $d(\tilde x,\tilde L^c)<\mu_1/C$ (where $C$ is given by Lemma~\ref{dc_lemma1a})
\end{itemize}

the intersection 
$W^u_{\mu_1}(\tilde x)\cap\left(\bigcup_{z \in \tilde L^c}W^s_{\mu_1}(\tilde 
z)\right)$ 
consists of a unique point. 
\label{dc_lemma2a}
\end{lemma}

The proof is identical to the proof of Lemma~\ref{dc_lemma2} just replacing the use of 
Remark~\ref{r.distance} by the following remark which is an analogous statement on the 
holonomy cover: 

\begin{rem} For any $\eta~>~0$  there is $\mu~>~0$ so that, for any center leaf $L$ and 
any $i$ so that $B(L,\delta)~\subset~V_i$,  for any lift $\tilde L$ in $\tilde V_i$, 
for any  $\tilde x,\tilde y~\in~\tilde L$ one has   
$$d(\tilde x,\tilde y)<\mu\Longrightarrow d^c(\tilde x,\tilde y)<(1+\eta)d(\tilde x,
\tilde y).$$
\end{rem}

\subsubsection{Projection distance}

\begin{lemma}\label{l.proj} There is $\mu_2~>~0$ so that for any given pair of points $x,y~\in~M$ with 
$d(x,y)~<~\mu_2$, for any given $i$ so that $B_H(W^c(x),\delta)~\subset~V_i$, for any given lift 
$\tilde x~\in~\tilde V_i$ the following statements hold: 
\begin{itemize}
\item There is a unique lift $\tilde y~\in~\tilde V_i$ so that 
$d(\tilde x,\tilde y)=d(x,y)$.
\item For any $\tilde w~\in~W^c(\tilde y)$ there is a unique point 
$$\Pi_{\tilde x,\tilde y}(w) \in W^u_{\mu_1}(\tilde w)\cap W^s_{\mu_1}(W^c(\tilde x)).
$$ 
\item The map $w~\mapsto~\Pi_{\tilde{x},\tilde{y}}(w)$ is continuous. One denotes by 
$$\De^u_{proj}(\tilde x,\tilde y)= \sup_{\tilde w\in W^c(\tilde y) }d^u(\tilde w,\Pi_{\tilde x,
\tilde y}(w))$$ the \emph{unstable projection distance}.
\item The distance $\De^u_{proj}(\tilde x,\tilde y)$  depends neither on the lift 
$\tilde x~\in~\tilde V_i$ nor on $i$ so that $B_H(W^c(x),\delta)~\subset~V_i$ 
(hence, it depends only on $x,y$).  We denote $\De^u_{proj}(x,y)=\De^u_{proj}(\tilde x,\tilde y).$

\end{itemize}
\end{lemma} 
\begin{proof} The uniqueness  of the lift $\tilde y$ nearby $\tilde x$ comes from the 
fact that the lift is a local isometry on $B(\tilde x,\delta)$.

For the existence and uniqueness of the intersection point  $W^u_\mu(\tilde w))\cap 
W^s_\mu(W^c(\tilde x))$, $\tilde w~\in~W^c(\tilde y)$, it suffices to choose $\mu$ 
small enough so that $W^c(\tilde y)~\subset~B_H(W^c(\tilde x),\mu_1/C)$ where $\mu_1$ and $C$ are 
given by Lemma~\ref{dc_lemma2a}. 

The continuity of $\De^u_{proj}(\tilde x,\tilde y)$ comes from the compactness of $W^u_{\mu_1}
(W^c(\tilde y))\cap W^s_{\mu_1}(W^c(\tilde x))$ together with the existence and  
uniqueness  of the intersection point $W^u_{\mu_1}(\tilde w)\cap W^s_{\mu_1}
(W^c(\tilde x))$, $\tilde w~\in~W^c(\tilde y)$.

The fact that $\De^u_{proj}(x,y)$ does not depend on the lifts comes from 
Lemma~\ref{lemma_independence_cover}: the $\delta_0$-neighborhoods of the lifts of 
$W^c(x)$ on $\tilde V_i$, for $B_H(W^c(c),\delta)~\subset~V_i$ are all isometric by 
isometries conjugating the lifted foliations. 

\end{proof}

\begin{rem}\label{r.proj}
\begin{itemize}
\item The Lemma~\ref{l.proj} above assures also the existence of a unique intersection point of 
the local stable manifold $W^s_{\mu_1}(\tilde{w})$ for any $\tilde{w}~\in~W^c(\tilde{y})$ with 
the local center unstable manifold $W^u_{\mu_1}(W^c(\tilde{x}))$ for any two nearby points 
$x,y$ with $d(x,y)~<~\mu_2$ if lifted to a holonomy cover. We can therefore analogously define a 
\emph{stable projection distance} by $\Delta^s_{proj}(x,y)$.
\item For $x,y$ with $d(x,y)~<~\mu_2$ (for $\mu_2$ given by Lemma~\ref{l.proj} and $\mu_1$ given by Lemma~\ref{dc_lemma2a}) one has 
$$\De^u_{proj}(x,y)=0\Longleftrightarrow W^c(y)\subset W^s_{\mu_1}(W^c(x)).$$ 
As a consequence, the foliations $\cF^s$ and $\cF^c$ are jointly integrable (that is, 
there is a center stable foliation subfoliated by $\cF^s$ and $\cF^c$) if and only if  
there is $\mu~>~0$ so that for every $x,y$ with $y~\in~W^s_\mu(W^c(x))$ one has 
$\De^u_{proj}(x,y)=0$. The analogous statement holds for $\Delta^s_{proj}$ and the center unstable manifold. 
\item By construction it is $\De^u_{proj}(x,y)~<~\mu_1$ and $\Delta^s_{proj}(x,y)~<~\mu_1$.
\item The unstable projection distance $\Delta^u_{proj}(x,y)$ is not a distance: it is not symmetric and may not satisfy the 
triangle inequality.
\end{itemize}
\end{rem}

\subsubsection{The projection distance and the dynamics}
Recall that $\beta~>~1$ denotes the minimum expansion of $Df$ 
on the unstable bundle $E^u$. 

\begin{lemma}\label{l.Delta-expansion} There is $\mu~>~0$ so that for any given pair of points $x,y~\in~M$ 
with $d(x,y)~<~\mu$ and $d(f(x),f(y))~<~\mu$ one has $\De^u_{proj}(f(x),f(y))~\geq~\beta\De^u_{proj}(x,y)$.
\end{lemma}
\begin{proof}
Fix $\mu_2~>~0$ of Lemma~\ref{l.proj} and $C~>~0$ of Lemma~\ref{dc_lemma1a}. Choose $\mu~>~0$ smaller than $\mu_2$ and such that $\mu(1 + 7 \lambda C + 3 \lambda C^2)~<~\delta_0$ with $\delta_0~>~0$ of 
Lemma~\ref{lemma_independence_cover} assuring that the holonomy 
covering maps are local isometries in our setting. Recall that $\lambda~>~1$ is the upper bound of $\left\| Df\right\|$. \\ 
Let $x,y~\in~M$ with $d(x,y)~<~\mu$ and $d(f(x),f(y))~<~\mu$. Let $i,j$ such that 
$B_H(W^c(x),\delta_0)~\subset~V_i$ and  $B_H(W^c(f(x)),\delta_0)~\subset~V_j$. Then, for any lift
$\tilde{x}~\in~p^{-1}_i(x)$, Lemma~\ref{l.proj} implies the existence of a unique lift 
$\tilde{y}~\in~p_i^{-1}(y)~\subset~\tilde{V}_i$ such that 
$d(\tilde{x},\tilde{y})~=~d(x,y)$.

Lemma~\ref{l.proj} asserts that, for any 
$\tilde{w}~\in~W^c(\tilde{y})$ there are
\begin{itemize}
 \item a unique intersection point  $\tilde{z}_{\tilde w} = W^u_{\mu_1}(\tilde{w}) \cap W^s_{\mu_1}(W^c(\tilde{x}))$ and 
\item a unique point $\tilde u_{\tilde w}~\in~W^c(\tilde x)$ so that $\tilde z_{\tilde w}~\in~W^s_{\mu_1}(\tilde u_{\tilde w})$.
\end{itemize}
Because $p_i$ is a local isometry,  the projections $w=p_i(\tilde w)\in W^c(y)$ ,  $z_{\tilde w}=p_i(\tilde{z}_{\tilde w})$, 
and  $u_{\tilde w}=p_i(\tilde{u}_{\tilde w})$
satisfy:
$$d^u(w, z_{\tilde w})=d^u(\tilde w,\tilde z_{\tilde w}) \mbox{ and } d^s( z_{\tilde w}, u_{\tilde w})=d^s(\tilde z_{\tilde w}, \tilde u_{\tilde w}).$$

Analogously, for any $\widetilde{f(x)}~\in~p^{-1}_j(y)$ there exists a unique lift 
$\widetilde{f(y)}~\in~p_j^{-1}(f(y))~\subset~\tilde{V}_j$, and for any 
$\tilde{v}~\in~W^c(\widetilde{f(y)})$ there exist
\begin{itemize}
 \item a unique intersection point and
$\tilde{z}_{\tilde v} = W^u_{\mu_1}(\tilde{v}) \cap W^s_{\mu_1}(W^c(\widetilde{f(x)}))$. 
\item a unique point $\tilde u_{\tilde v}~\in~W^c(\widetilde{f(x)})$ whose local stable manifold contains $\tilde{z}_{\tilde v}$.
\end{itemize}
Once again we denote by $v~\in~W^c(f(y))$, $z_{\tilde v}$ and $u_{\tilde v}$ their projection by the local isometry $p_j$.

Recall that the projection distances are defined by 
$$\Delta^u_{proj}(x,y) =\sup_{\tilde w\in W^c(\tilde y)} \{d^u(\tilde{w},\tilde{z}_{\tilde w})\}
\mbox{ and } 
\Delta^u_{proj}(f(x),f(y)) =\sup_{\tilde v\in W^c(\widetilde{f(y)})} \{d^u(\tilde{v},\tilde{z}_{\tilde v})\}.
$$
Notice that, for any $\tilde w~\in~W^c(\tilde y)$, one has
$$d^u(f(w),f(z_{\tilde w}))\geq \beta d^u(w,z_{\tilde w}).$$
Using the fact that the projections $p_i$ and $p_j$ are local isometies,  we get that $\Delta^u_{proj}(f(x),f(y))~\geq~\beta\Delta^u_{proj}(x,y)$ 
(concluding the proof of Lemma~\ref{l.Delta-expansion}) if we prove the following claim:
\begin{claim} \label{c.projection}
For any $\tilde w~\in~W^c(\tilde y)$ there is $\tilde v~\in~W^c(\widetilde{f(y)})$ so that 
$$f(z_{\tilde w})=z_{\tilde v}.$$
\end{claim}

\begin{proof}[Proof of Claim~\ref{c.projection}]
We start by proving $f(z_{\tilde y})=z_{\widetilde{f(y)}}$. For that we will use the fact that $p_j$ is a local isometry for lifting the point 
$f(z_{\tilde y})$ and show that this point belongs to the local centerstable manifold of $\widetilde{f(x)}$ 
and to the local unstable manifold of $\widetilde{f(y)}$.

One has 
\begin{itemize}
\item $d(\tilde x,\tilde y)=d(x,y)<\mu$,
\item $d^s(\tilde{z}_{\tilde y},\tilde u_{\tilde y}) < 2C\mu$ according to  Lemma~\ref{dc_lemma2a},
\item $d^u(y,z_{\tilde y})=d^u(\tilde{y},\tilde{z}_{\tilde y})~<~C\mu$ and as a consequence
\item $d(y,u_{\tilde y})=d(\tilde y, \tilde u_{\tilde y})~<~3C\mu$ and by the triangle inequality
\item $d^c(x,u_{\tilde y})=d^c(\tilde x,\tilde u_{\tilde y})~<~C(3C\mu+\mu).$
\end{itemize}
Recalling that $\|Df\|$ is bounded by $\lambda$ one gets
\begin{itemize}
\item $d(f(x),f(y))~<~\mu$ (by hypothesis), 
\item $d^s(f(z_{\tilde y}),f(u_{\tilde y}))~<~2\lambda C\mu$,
\item $d^u(f(y),f(z_{\tilde y}))~<~\lambda C\mu$,
\item $d(f(y),f(u_{\tilde y}))~<~3\lambda C\mu$ and
\item $d^c(f(x),f(u_{\tilde y}))~<~\lambda C(3C\mu+\mu)$.
\end{itemize}

We have chosen $\mu$ so that the sum of all these distances is less than the isometry radius $\delta_0$ of the cover. 
As $d^c(f(x), f(u_{\tilde y}))~<~\delta_0$, there is a unique lift $\widetilde{f(u_{\tilde y})}$ so that 
$$d^c(\widetilde{f(x)},\widetilde{ f(u_{\tilde y})})= d^c(f(x), f(u_{\tilde y})).$$

In the same way there are unique lifts $\widetilde{f(z_{\tilde y})}$ and $\widehat{f(y)}$ so that
\begin{align*} 
d^s(\widetilde{f(z_{\tilde y})},\widetilde{f(u_{\tilde y})})&= d^s(f(z_{\tilde y}),f(u_{\tilde y}))\mbox{, and }\\ 
d^u(\widehat{f(y)},\widetilde{f(z_{\tilde y})})&= d^u(f(y),f(z_{\tilde y})). 
\end{align*}
By construction, this means that $\widetilde{f(z_{\tilde y})}$ is the projection of $\widehat{f(y)}$  on the centerstable leaf of $\widetilde{f(x)}$.

One deduces that $\widehat{f(y)}$ is a lift of $y$ with $d(\widetilde{f(x)},\widehat{f(y)})~<~\delta_0$. 
But $\widetilde{f(y)}$ is the unique lift of $y$ at distance less than $\delta_0$ from $\widetilde{f(x)}$ so that 
$$\widehat{f(y)}=\widetilde{f(y)}.$$ 
As a consequence it is
$$\widetilde{f(z_{\tilde y})}=\tilde z_{\widetilde{f(y)}}$$ and thus we proved the claim for $\tilde w=\tilde y$:
$$f(z_{\tilde y})= z_{\widetilde{f(y)}}.$$ 

Now consider any point $\tilde w~\in~W^c(\tilde y)$. As $W^c(\tilde y)$ is a connected manifold one can fix a path
$\tilde w_t~\in~W^c(\tilde y)$, $t~\in~[0,1]$, with $\tilde w_0~=~\tilde y$ and $\tilde w_1~=~\tilde w$.

The points  $u_{\tilde w_t}~\in~W^c(x)$  and $z_{\tilde w_t}~\in~W^u(w_t)\cap W^s(u_{\tilde w_t}) $ depend continuously on $t~\in~[0,1]$ and satisfy
\begin{itemize}
\item  $d^s(z_{\tilde w_t}, u_{\tilde w_t})=d^s(\tilde{z}_{\tilde w_t},\tilde u_{\tilde w_t})~\leq~2C\mu,$
\item $d^u(w_t,z_{\tilde w_y})=d^u(\tilde{w_t},\tilde{z}_{\tilde w_t})~<~C\mu$, and as a consequence
\item $d(w_t,u_{\tilde w_t})=d(\tilde w_t, \tilde u_{\tilde w_t})~<~3C\mu.$
\end{itemize}

Applying $f$ we get once more that the distance between $f(w_t)$, $f(z_{\tilde w_t})$ and $f(u_{\tilde w_t})$ 
remain smaller that $\delta_0$ so that the choice of a lift of $f(w_t)$ determines a lift of $f(z_{\tilde w_t})$ and $f(u_{\tilde w_t})$ .

Notice that one can make a continuous choice of lifts $\widetilde{f(w_t)}$ with $\widetilde{f(w_0)}= \widetilde{f(y)}$: this is because $f$ 
conjugates the holonomy representation of the leaves $W^c(y)$ and $W^c(f(y))$ and that $p_i$ and $p_j$ 
induces the holonomy cover over $W^c(y)$ and $W^c(f(y))$.

This continuous lift $\widetilde{f(w_t)}$ induces continuous lifts $\widetilde{f(z_{\tilde w_t})}$ and $\widetilde{f(u_{\tilde w_t})}$ 
which coincide with 
$\widetilde{f(z_{\tilde y})}$ and $\widetilde{f(u_{\tilde y})}$ for $t=0$.  In particular, $\widetilde{f(u_{\tilde w_0})}$ belongs to 
$W^c(\widetilde{f(x)})$ and therefore  $\widetilde{f(u_{\tilde w_t})}$ belongs to 
$W^c(\widetilde{f(x)})$.  One deduces that, for all $t~\in~[0,1]$ one has
\begin{align*} 
\widetilde{f(z_{\tilde w_t})}&=\tilde z_{\widetilde{f(w_t)}}\mbox{ so that }\\
f(z_{\tilde w})&=z_{\widetilde{f(w_1)}} \mbox{ with } \widetilde{f(w_1)}\in W^c(\widetilde{f(y)}).
\end{align*}
This proves the claim.
\end{proof}
This concludes the proof of Lemma~\ref{l.Delta-expansion}.
\end{proof}

\subsection{Dynamical coherence: proof of Theorem~\ref{theorem_dc}}

We have now the tools for making rigorous the intuitive idea for proving the dynamical coherence: given two 
center leaves through points $w$ and $z$ in the same stable leaf, the distance between the images by 
$f^n$ of the leaves tends to $0$: one would like to conclude that every point of one leaf belongs to 
the stable 
leaf of a point in the other leaf. In our language, this means that the unstable projection distances 
$\De^u_{proj}(w,z)$ and $\De^u_{proj}(z,w)$ vanish. This is obtained in Corollary~\ref{c.Delta-expansion} 
with the help of  Lemma~\ref{l.Delta-expansion}.  

Indeed, we show
\footnote{In a previous version of this paper, 
we proved the dynamical coherence without noticing that our proof also gave the completeness. 
We rewrote the proof (after reading \cite{C10} and another unpublished version of the 
article \cite{C11}) emphasizing that property.} 
a stronger property than dynamical coherence, called \emph{completeness} 
in \cite{C10}.
\begin{prop}\label{p.complete} For any $x~\in~M$ one has 
 $$\bigcup_{w\in W^s(x)} W^c(w) =   \bigcup_{z\in W^c(x)} W^s(z).$$
\end{prop}

Notice that Proposition~\ref{p.complete} allows us to define the center stable leaves through a point $x\in M$ as 
$W^{cs}(x):=\bigcup_{w\in W^s(x)} W^c(w) = \bigcup_{z\in W^c(x)} W^s(z)$. The center stable foliation 
we obtain is $f$-invariant
as the center foliation $\cF^c$ is assumed to be invariant. 
Applying  Proposition~\ref{p.complete}
to $f^{-1}$ one gets the  (invariant) center unstable foliation, and thus the dynamical coherence.  In other words 
Proposition~\ref{p.complete} implies Theorem~\ref{theorem_dc}. 
The proof of this proposition is the aim of this section.\\
We start the proof with the following corollary: 
\begin{corol} \label{c.Delta-expansion} Let $\mu$ be the constant chosen in Lemma~\ref{l.Delta-expansion}. 
The three following properties are equivalent:
\begin{itemize}
\item The two center leaves $W^c(x)$ and $W^c(y)$ satisfy  that, 
there is $n_0$ and $w~\in~W^c(x)$ and $z~\in~W^c(y)$ such that for any $n~\geq~n_0$  the distance 
$d(f^n(w),f^n(z))$ is bounded by $\mu$. 
\item  There is $z~\in~W^c(y)$ so that $z$ belongs to  $\bigcup_{w\in W^c(x)} W^s(w)$.
\item  $W^c(y)~\subset~\bigcup_{w\in W^c(x)} W^s(w)$.
\item The two center leaves $W^c(x)$ and $W^c(y)$ satisfy  that, 
there is  $w~\in~W^c(x)$ and $z~\in~W^c(y)$ such that $\De^u_{proj}(w,z)$ is well defined and
$\De^u_{proj}(w,z)~=~0$.
\end{itemize}

%In the same way, if for any $n\geq n_0$ the Hausdorff distance 
%$d_H(f^{-n}(W^c(x)),f^{-n}(W^c(y))$ is bounded by $\mu$ then this 
%is equivalent to that $W^c(y)$ belongs to  $\bigcup_{z\in W^c(x)} W^u(z)$.
 
\end{corol}
\begin{proof} The fact that the third item implies the second is trivial.  
The fact that the second item implies the first one is also straightforward.

Let us first show that the first item implies the fourth one. 
According to Lemma~\ref{l.proj} and Remark~\ref{r.proj} (because $\mu$ is chosen smaller than $\mu_2$) one gets that 
$\De^u_{proj}(f^n(w),f^n(z))$ is well defined for every $n~\geq~n_0$ and bounded by $\mu_1$. 
However, Lemma~\ref{l.Delta-expansion} implies that
$$\De^u_{proj}(f^n(w), f^n(z))~\geq~\beta^n\De^u_{proj}(w,z),$$
for every $n~\geq~0$, where $\beta~>~1$ is the expansivity constant of 
$Df$ on $E^u$.  As a direct consequence it follows $\De^u_{proj}(w,z)~=~0$, the forth item.
This implies the third item: 
$$W^c(y)=W^c(z)~\subset~\bigcup_{u\in W^c(w)} W^s(u)=\bigcup_{w\in W^c(x)} W^s(x).$$
This ends the proof.
\end{proof}

As a direct consequence of Corollary~\ref{c.Delta-expansion} one gets 

\begin{corol} For any $y$ such that  $W^c(y)\cap~\bigcup_{z\in W^c(x)} W^s(z)~\neq~\emptyset$ one has:
$$\bigcup_{z\in W^c(x)} W^s(z)=\bigcup_{w\in W^c(y)} W^s(w).$$
\end{corol}
\begin{proof} First, Corollary~\ref{c.Delta-expansion} implies that 
$W^c(y)~\subset~\bigcup_{z\in W^c(x)} W^s(z)$.
Then the stable manifold through any point $w~\in~W^c(y)$ is contained in the 
stable manifold of a point $z~\in~W^c(x)$ that is 
$$ \bigcup_{w\in W^c(y)} W^s(w)\subset \bigcup_{z\in W^c(x)} W^s(z).$$

For getting the reverse inclusion it is now enough to prove  
\begin{equation}\label{e.inclusion}
W^c(x)~\cap~\bigcup_{w\in W^c(y)} W^s(w)~\neq~\emptyset.
\end{equation}
This is just because the hypothesis means there is $w~\in~W^c(y)$ and $z~\in~W^c(x)$ so that $W^s(w)=W^s(z)$. 
Thus $z$ belongs to  $\bigcup_{w\in W^c(y)} W^s(w)$, which is therefore not empty, proving Equation~\ref{e.inclusion}.

\end{proof}

\begin{corol} For any $x~\in~M$ one has
$$\bigcup_{w\in W^s(x)} W^c(w)\subset   \bigcup_{z\in W^c(x)} W^s(z).$$
\end{corol}

\begin{proof} If $y~\in~W^c(w)$ for some $w~\in~W^s(x)$, then $d(f^n(w), f^n(x))~\to~0$ as $n~\to~+\infty$ 
 so that Corollary~\ref{c.Delta-expansion} 
implies that 
$W^c(y)=W^c(w)~\subset~\bigcup_{z\in W^c(x)} W^s(z)$.
\end{proof}

We can now finish the proof of the dynamical coherence (and therefore of Theorem~\ref{theorem_dc}) 
by proving:

\begin{proof}[Proof of Proposition~\ref{p.complete}]
 We just have to prove $\bigcup_{z\in W^c(x)} W^s(z)~\subset~\bigcup_{w\in W^s(x)} W^c(w)$.  

 Let $y~\in~\bigcup_{z\in W^c(x)} W^s(z)$, that is: there is $z~\in~W^c(x)$ with $y~\in~W^s(z)$. So for 
 $n$ large $d(f^n(y),f^n(z))$ remains bounded by $\mu$.
 
 According to Corollary~\ref{c.Delta-expansion}, one deduces  $W^c(x)~\subset~\bigcup_{w\in W^c(y)} W^s(w)$. 
In particular, there is a $w~\in~W^c(y)$ so that $x~\in~W^s(w)$, that is $w~\in~W^s(x)$ and $y~\in~W^c(w)$,
which concludes the proof. 
\end{proof}

\subsection{Expansivity in the case of trivial center-holonomy}
The expansivity of $f$ with respect to orbits of center leaves follows almost 
immediately from previous arguments for the proof of dynamical coherence 
if we assume a compact center foliation with 
\emph{trivial holonomy}:
\begin{prop}\label{p.expansive} Le $f$ be  a partially hyperbolic diffeomorphism having 
a uniformly compact invariant center foliation $\cF^c$
without holonomy.  Then the homeomorphism \newline$F:M/\cF^c~\rightarrow~M/\cF^c$ induced by 
$f$ on the quotient of the center foliation is expansive.
\end{prop}
\begin{rem}
\begin{itemize}
\item The expansivity on the quotient space implies the plaque expansivity of the 
center foliation $(f,\cF^c)$, hence it implies its structural stability 
(see discussion of these properties in the introduction of this article). 

\item Let us recall and emphasize that the conclusion of Proposition~\ref{p.expansive} 
is wrong without the assumption of trivial holonomy. See an example in \cite{BoW05}, further explained in Section~\ref{sec:comments}. 
\end{itemize} 
\end{rem}
The proof of Proposition~\ref{p.expansive} uses the following lemma:

\begin{lemma}\label{l.uniqueDelta} Le $f$ be  a partially hyperbolic diffeomorphism having 
a uniformly compact invariant center foliation $\cF^c$ without holonomy.  
Then there is $\mu>0$ so that, for every pair of center leaves $W^c_1$ $W^c_2$ whose Hausdorff distance is 
less than $\mu$ then 
for every pair of  pairs $(x_1,x_2),(y_1,y_2)\in W^c_1\times W^c_2$ with  
$d(x_1,x_2)<\mu$ and $d(y_1,y_2)<\mu$ one has  
$$\De^u_{proj}(x_1,x_2)=\De^u_{proj}(y_1,y_2)\mbox{ and }\De^s_{proj}(x_1,x_2)=\De^s_{proj}(y_1,y_2) .$$
\end{lemma}
\begin{proof}[Idea of the proof] Let $\mu_1,\mu_2$ be defined by  
Lemma~\ref{dc_lemma2a} and~\ref{l.proj}. As we assume trivial holonomy, 
there is $\mu>0$ so that for every leaf 
$W^c_1$,$W^c_2$ with 
Hausdorff distance less than $\mu$  for every $w\in W^c_2$ there is a unique point 
$\Pi(w)\in W^u_{\mu_1}(w)\cap W^s_{\mu_1}(W^c_1)$:  indeed,  that is the second claim in Lemma~\ref{l.proj}, 
noticing that the holonomy covers are trivial. 

Therefore for every $(x_1,x_2)\in W^c_1\times W^c_2$ with  $d(x_1,x_2)<\mu$ one has 
$$\De^u_{proj}(x_1,x_2)=\sup_{w\in W^c_2} d^u(w,\Pi(w)),$$
which finishes the proof.
\end{proof}

According to Lemma~\ref{l.uniqueDelta}, there is no ambiguity in denoting 
$$\De^u_{proj}(W^c_1,W^c_2):=\De^u_{proj}(x_1,x_2).$$

Now we can prove Proposition~\ref{p.expansive}:
\begin{proof}[Proposition~\ref{p.expansive}]

We choose $\mu~>~0$ small enough so that it satisfies both  Lemma~\ref{l.uniqueDelta} and  
Lemma~\ref{l.Delta-expansion}. We consider two arbitrary center leaves 
$W^c(x)$ and $W^c(y)$ which remain for all iterates at a distance less than $\mu~>~0$, 
i.e. $d_H(W^c(f^nx),W^c(f^ny))~<~\mu$ for all $n~\in~\mathbb{Z}$. 

So, first we have that $d_H(W^c(f^nx),W^c(f^ny))~<~\mu$ for all $n~\geq~0$. 
Applying Corollary~\ref{c.Delta-expansion}  we get $\De^u_{proj}(W^c(x),W^c(y))=0$ and 
$W^c(y)~\subset~W^s_{\mu_1}(W^c(x))$. This implies that $y$  is the unique intersection point  of
$W^u_{\mu_1}(y)$ with  $ W^s_{\mu_1}(W^c(x))$, that is: 
$$\left\{y\right\}=W^u_{\mu_1}(y) \cap W^s_{\mu_1}(W^c(x)).$$

Now we consider the backward iterates, applying Corollary~\ref{c.Delta-expansion} for $f^{-1}$: 
 we get $\De^s_{proj}(W^c(x),W^c(y))=0$ and 
$W^c(y)~\subset~W^u_{\mu_1}(W^c(x))$, and thus:
$$\left\{y\right\}=W^s_{\mu_1}(y) \cap W^u_{\mu_1}(W^c(x)).$$

So there are $w,z\in W^c(x)$ so that $y\in W^s_{\mu_1}(z)$ and  $y\in W^u_{\mu_1}(w)$. 
Thus, according to the notation in the proof of Lemma~\ref{l.uniqueDelta}, $z$ is the projection 
$\Pi(w)$ of $w$ on $W^s(W^c(x))$.  This implies $z=w$. So $y\in W^s_{\mu_1}(w)\cap W^u_{\mu_1}(w)=\{w\}$.
 
 This implies $y~\in~W^c(x)$ concluding the proof.

\end{proof}

\section{Shadowing Lemma}\label{sec:shadowing}

In this section we prove the Shadowing Lemma (Theorem~\ref{shadowing_lemma}) 
on the leaf space $M/\mathcal{F}^c$ for the homeomorphism $F$ induced by a partially 
hyperbolic $C^1$-diffeomorphism $f:M~\rightarrow~M$ with an invariant uniformly compact 
center foliation $\mathcal{F}^c$.
\subsection{Some remarks concerning the proof.}
First, recall that we defined a modified Hausdorff distance $\Delta_H$ on the space of center leaves which is topologically equivalent to the Hausdorff distance, see Corollary~\ref{corol_distance}. Therefore, it is enough to prove the 
Shadowing Lemma for $\De_H$. 
\begin{rem} If a homeomorphism $h$ of a compact metric space $K$ satisfies the shadowing property for finite 
positive pseudo orbits, then it satisfies the shadowing property for every bi-infinite pseudo orbits. 
\begin{proof}
Consider $\delta>0$ and $\varepsilon>0$ so that every finite positive $\varepsilon$-pseudo orbit 
is $\delta$-shadowed. Let $(x_i)_{i\in\ZZ}$ be a bi-infinite $\varepsilon$-pseudo orbit. 
For every $n>0$, let $z^n\in K$ so that its orbit shadows the finite 
positive $\varepsilon$-pseudo-orbit $(x_{-n}, \dots, x_n)$; 
let $y^n$ denote $h^n(z^n)$.  By definition, $d(h^i(y^n),x_i)\leq \delta$ for every $i\in\{-n,\dots,n\}$. 
By compactness of $K$ one can extract a subsequence $y^{n_j}$ converging to some $y\in K$. \\
By continuity of $h^i$ for every $i$ one gets $d(h^i(y),x_i)\leq \delta$ for every $i\in\ZZ$, ending the proof. 
\end{proof}
\end{rem}
Note furthermore that the shadowing property for a homeomorphism $h$ 
of a compact metric space is equivalent to 
the shadowing property for some $h^N$, $N>0$. 

Thus, up to replace $f$ by some iterate $f^N$, we may assume 
$$2C\alpha<1,$$ 
where  $\alpha$ is an upper bound of $\|Df|_{E^s}\|$ and of $\| Df^{-1}|_{E^u}\|$, and $C$ is the constant given by Lemma~\ref{dc_lemma1a}. 
Recall that $\delta_0>0$ is given by Lemma~\ref{lemma_independence_cover} and assures that within a Hausdorff ball of radius $\delta_0$ the holonomy covering maps are isometries and that the distances are independent from the choice of the holonomy cover. 

The aim of this section is the proof of Proposition~\ref{p.shadowing} below which provides the proof for the Shadowing Lemma:

\begin{prop}\label{p.shadowing} With the notations above,
for any $\eta~>~0$ consider $\varepsilon~<~\min\{ \frac{(1-\alpha)\eta}{2C}, \delta_0 \}$.  
Then for any  $\varepsilon$-pseudo orbit $\{W_i\}_{i \geq 0}$ for $\Delta_H$, 
where $W_i$ is a center leaf of $f$,
there is a center leaf $W$ so that for any $i~\geq~0$ one has
$$\Delta_H(f^i(W), W_i)\leq \eta.$$
 
\end{prop}

\subsection{Adaption of the stable and unstable projection distances.} 

According to Theorem~\ref{theorem_dc}, we know that $f$ is dynamically coherent so that one can use 
the center stable and center unstable foliations.  We will denote by $W^s(W)$ and $W^u(W)$ 
the center stable 
and center unstable leaves of a given center leaf $W$.  We have seen (see  Proposition~\ref{p.complete}) 
that $W^s(W)$ 
is the union of the  stable leaves through $W$.  We will denote by $W^s_\mu(W)$ the union of the local
stable manifolds through the points of $W$. \\
In the whole proof, we use intensively the notion of lifts of center leaves.  
This refers to lifts on a holonomy cover $\{U_i,p_i\}$  for which a $\delta_0$-neighborhood of 
the considered leaves are contained in $U_i$. The local isometry property 
(see Lemma~\ref{lemma_independence_cover}) implies that 
the quantities we define are independent from the chosen $\left\{U_i,p_i\right\}$. 
However, each leaf may have several lifts to
$\tilde U_i$ and the estimates for a pair of leaves depend on the lift we choose for each of them. 
For this reason we will often use the expression \emph{given two leaves $W_1$ and $W_2$, 
and given a lift
$\widetilde  W_1$, there is a lift $\widetilde W_2$ with the following property...}. \\

In the section above we established the unstable projection distance for any pair of points $x,y~\in~M$ 
at a distance bounded by $\mu~>~0$ given by Lemma~\ref{l.proj}. Let $(U_i,p_i)$ 
be any neighborhood of a holonomy cover such that $B_H(W^c(x),\delta_0)~\subset~U_i$ and 
let $\tilde{x}~\in~p^{-1}_i(x),\tilde{y}~\in~p^{-1}_i(y)$ be any lifted points with 
$d(\tilde{x},\tilde{y})~=~d(x,y)~<~\mu$, then we recall the following definition from Lemma~\ref{l.proj}: 
\begin{align*}
\Delta^u_{proj}(x,y)&=\sup_{\tilde{w}\in W^c(\tilde{y})}d^u(\tilde{w},
\Pi_{\tilde{x},\tilde{y}}(w))\;\mbox{where}\\
\Pi_{\tilde x,\tilde y}(w)&= W^u_{\mu_1}(\tilde w)\cap W^s_{\mu_1}(W^c(\tilde x))\;
\mbox{for all}\;\tilde w~\in~W^c(\tilde y). 
\end{align*}
This definition is independent from the points lifted to a holonomy cover as long as 
$d_H(\widetilde{W_1},\widetilde{W_2})~<~\mu$ for any choice of two center leaves. 
So we can denote the unstable 
projection distance between lifted center leaves by  
$$\Delta^u_{proj}(\widetilde{W_1},\widetilde{W_2}) = \Delta^u_{proj}(w_1,w_2)$$
where $\widetilde{W_1}$ and $\widetilde{W_2}$ are those lifts of $W_1$ and $W_2$ such that 
$d_H(\widetilde{W_1},\widetilde{W_2})~<~\mu$ 
and $w_1~\in~W_1,w_2~\in~W_2$ are 
any points with $d(w_1,w_2)~<~\mu$. The center leaf $\widetilde{W_2}$ is projected along the unstable 
foliation onto the center stable leaf $W^s(\widetilde{W_1})$. In the case of the stable projection 
distance the center leaf $\widetilde{W_1}$ is projected along the stable foliation onto the center 
unstable leaf $W^u(\widetilde{W_2})$. As the projection distance is not symmetric, we write the stable 
projection distance as
$$\Delta^s_{proj}(\widetilde{W_2},\widetilde{W_1})=\Delta^s_{proj}(w_2,w_1).$$

We recall (see Remark~\ref{r.proj}) that the projection distance
\begin{itemize}
 \item depends on the pair of lifts $\widetilde{W_1}$ and $\widetilde{W_2}$ 
 and has no meaning for the leaves $W_1$ and $W_2$;
 \item even for the lifts, it is not a distance: it is not symmetric 
 (in general, $\Delta^u_{proj}(\widetilde{W_1},\widetilde{W_2})~\neq~\Delta^u_{proj}(\widetilde{W_2},
 \widetilde{W_1})$),
 and it does not satisfy the triangle inequality. 
\end{itemize}

Now we can formulate the following lemma to compare the Hausdorff distance with the unstable and 
stable projection distances on the holonomy cover:  
\begin{lemma} Let $\mu~>~0$ be given by Lemma~\ref{l.proj} and $C~>~0$ by Lemma~\ref{dc_lemma1a}. For all 
center leaves $W_1,W_2$ with $\Delta_H(W_1,W_2)~<~\mu$ there exist lifts 
$\widetilde{W_1}, \widetilde{W_2}$ such that 
\begin{align*}
\frac{1}{C} \Delta^u_{proj}(\widetilde{W_1},\widetilde{W_2})
&\leq d_H(\widetilde{W_1},\widetilde{W_2}),\\
\frac{1}{C}\Delta^s_{proj}(\widetilde{W_2},\widetilde{W_1}) 
&\leq d_H(\widetilde{W_1},\widetilde{W_2}), \\
d_H(\widetilde{W_1},\widetilde{W_2})&\leq \Delta^s_{proj}(\widetilde{W_2},\widetilde{W_1}) 
+ \Delta^u_{proj}(\widetilde{W_1},\widetilde{W_2}).
\end{align*}
\label{l.metric}
\end{lemma}
\begin{proof}
The first inequality comes then from Lemma~\ref{l.proj}: for any pair of points $(\tilde{x},\tilde y)~\in~\widetilde{W_1}\times\widetilde{W_2}$ with $d(\tilde x, \tilde y)~<~d_H(\widetilde{W_1},\widetilde{W_2})$ one has
$d^u(\tilde y, \Pi_{\tilde x,\tilde y}(\tilde y))~<~Cd(\tilde x,\tilde y)$. The second inequality is 
proved in a similar way.\\
The last inequality is implied by the triangle inequality. Let $\mu_1 > 0$ be given by Lemma~\ref{dc_lemma2a}. Given $\tilde y~\in~\widetilde{W_2}$ 
there is a unique $\tilde x~\in~\widetilde{W_1}$ so that 
$W^s_{\mu_1}(\tilde x)$ cuts $W^u_{\mu_1}(\tilde{y})$ in a point $\tilde w$.
Then $$d(\tilde x,\tilde y)<d^s(\tilde x,\tilde w)+d^u(\tilde w,\tilde y)
\leq\Delta^s_{proj}(\widetilde{W_2},\widetilde{W_1}) 
+ \Delta^u_{proj}(\widetilde{W_1},\widetilde{W_2}).$$
\end{proof}
\begin{rem}
If $\widetilde{W_1}~\subset~W^u_{\mu_1}(\widetilde{W_2})$, 
then $\Delta^s_{proj}(\widetilde{W_2},\widetilde{W_1}) =0$. 
Thus by Lemma~\ref{l.metric} one
has 
\begin{equation}\label{e.distances}
\frac{1}{C}\Delta^u_{proj}(\widetilde{W_1},\widetilde{W_2}) \leq d_H(\widetilde{W_1},\widetilde{W_2}) 
\leq \Delta^u_{proj}(\widetilde{W_1},\widetilde{W_2}).
\end{equation}
The analogous statement holds for the stable projection distance $\Delta_{proj}^s$. 
\end{rem}

As a direct consequence of Lemma~\ref{l.Delta-expansion} one gets: 
\begin{lemma} Let $\mu_1 > 0$ be given by Lemma~\ref{dc_lemma2a} and $\mu~>~0$ by Lemma~\ref{l.proj}. Let 
$W_1,W_2$ be any pair of center leaves and $\widetilde{W_1},\widetilde{W_2}$ be any lifts of these so that 
\begin{itemize}
 \item $\widetilde{W_1}~\subset~W^s_{\mu_1}(\widetilde{W_2})$,
 \item $\Delta^s_{proj}(\widetilde{W_1},\widetilde{W_2})~<~\mu$. 
\end{itemize}
Then for every $n~>~0$ there are lifts $\widetilde{f^n(W_1)}, \widetilde{f^n(W_2)}$ for which it holds:
\begin{itemize}
 \item $\widetilde{f^n(W_1)}~\subset~W^s_{\mu_1}(\widetilde{f^n(W_2)})$,
 \item $\Delta^s_{proj}(\widetilde{f^n(W_1)},\widetilde{f^n(W_2)})~<~\alpha^n \Delta^s_{proj}(\widetilde{W_1},\widetilde{W_2})$. 
\end{itemize}

One has a similar statement for iterates of  $f^{-1}$ and $\Delta^u_{proj}$.
\label{l.expansion}
\end{lemma}

\begin{proof}
Fix a lift $\widetilde{f^n(W_1)}$. Consider $\tilde x~\in~\widetilde{W_1}$ and 
let $\tilde y~\in~\widetilde{W_2}$ be the unique point so that 
$\tilde y~\in~W^s_{\mu_1}(\tilde x)$. Let $x$ and $y$ be their projections. 
Then, recalling the definition of the stable projection distance, 
one gets$$d^s(f^n(x),f^n(y))~\leq~\alpha^n d^s(x,y)~\leq~\alpha^n 
\Delta^s_{proj}(\widetilde{W_1},\widetilde{W_2}).$$
Let $\widetilde{f^n(x)}$ be a lift of $f^n(x)$ in  $\widetilde{f^n(W_1)}$. 
As $d(f^n(x),f^n(y))~\leq~\delta_0$, there is a unique lift $\widetilde{f^n(y)}$ so that
$d(\widetilde{f^n(x)},\widetilde{f^n(y)})~<~\mu$, and indeed by Lemma~\ref{l.proj} equal to $d(f^n(x),f^n(y))$. 
We denote by $\widetilde{f^n(W_2)}$ the lift of $f^n(W_2)$ through $\widetilde{f^n(y)}$. 

Let us show that the lift $\widetilde{f^n(W_2)}$ does not depend on the choice of $\tilde x$:
given any $\tilde x_1~\in~\widetilde{W_1}$, we fix a path $\tilde x_t$ inside $ \widetilde{W_1}$  
joining $\tilde x$ to $\tilde x_1$. 
Then the projections on $\widetilde{W_2}$ along the stable leaves define a path $\tilde y_t$ in 
$\widetilde{W_2}$ with 
$\tilde y_t~\in~W^s_{\mu_1}(\tilde x_t)$. 

Denote $x_t,y_t$ the projections of $\tilde x_t,\tilde y_t$. Then 
\begin{equation}\label{e.expansion}
 d^s(f^n(x_t),f^n(y_t))<\alpha^n d^s(\tilde x_t,\tilde y_t)
\leq \alpha^n\Delta^s_{proj}(\widetilde{W_1},\widetilde{W_2}) \leq \delta_0.
\end{equation}

Consider the  continuous lifts $\widetilde{f^n(x_t)}~\in~\widetilde{f^n(W_1)}$ so that  
$\widetilde{f^n(x_0)}=\widetilde{f^n(x)}$. Then we have a unique choice of a lift  $\widetilde{f^n(y_t)}$
at distance less than $\delta_0$ of  $\widetilde{f^n(x_t)}$, and these lifts define a path in the lifts of 
$f^n(W_2)$ starting at $\widetilde{f^n(y_0)}$.  Therefore this path is contained in $\widetilde{f^n(W_2)}$. 
This proves that  the lift $\widetilde{f^n(W_2)}$ is independent of the choice of $\tilde x\in\widetilde{W_1}$.

By construction we have  $\widetilde{f^n(W_1)}~\subset~W^s_{\mu_1}(\widetilde{f^n(W_2)})$, and the 
inequality (\ref{e.expansion}) implies
 $$\Delta^s_{proj}(\widetilde{f^n(W_1)},\widetilde{f^n(W_2)})  
 < \alpha^n \Delta^s_{proj}(\widetilde{W_1},\widetilde{W_2}).$$

\end{proof}

\begin{rem}\label{uniquelift}
From the considerations above and the proof of Lemma~\ref{l.Delta-expansion} we can easily deduce the following: 
If $\widetilde{W_1}$ and $\widetilde{W_2}$ are two lifted center leaves with 
$d_H(\widetilde{W_1}, \widetilde{W_2})<\mu/\lambda$ ($\mu > 0$ given by Lemma~\ref{l.proj} and $\lambda > 1$ given as the upper bound of $\left\|Df\right\|$)  and if $\widetilde{f(W_1)}$ is 
a lift of $f(W_1)$, then there is a unique lift $\widetilde{f(W_2)}$ with the following property: 

For every $\tilde x\in \widetilde{W_1}$ and $\tilde{y}\in\widetilde{W_2}$ with $d(\tilde x,\tilde y)<\mu/\lambda$
and if $\widetilde{f(x)}$ is a lift of $f(x)$ on $\widetilde{f(W_1)}$ then there is a lift 
$\widetilde{f(y)}\in\widetilde{f(W_2)}$ with $d(\widetilde{f(x)},\widetilde{f(y)})=d(f(x),f(y)).$

\end{rem}

\subsection{Proof of proposition~\ref{p.shadowing} (and of the shadowing lemma)}

\begin{figure}
\includegraphics[width=0.8\textwidth]{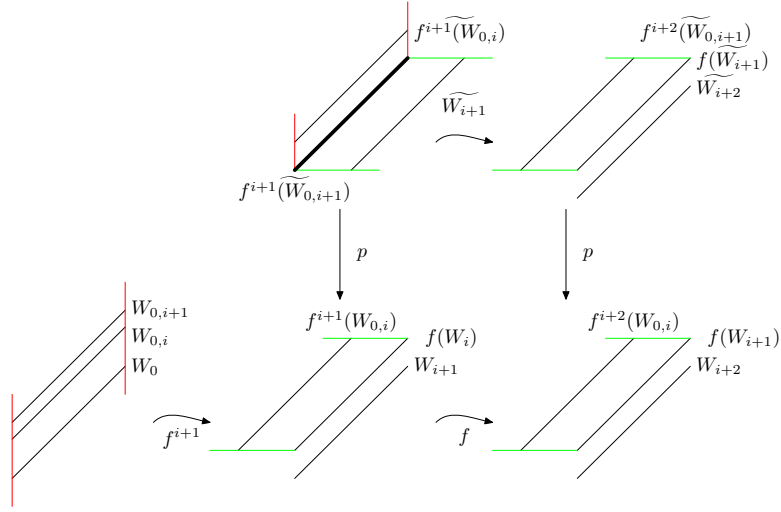}
\caption{Building the sequence $W_{0,i}$ by defining $\widetilde{f^{i+1}(W_{0,i+1})}$ in the holonomy cover.}
\label{fig:shadowing}
\end{figure}
\begin{proof} 
Let $\left\{W_i\right\}_{i \geq 0}$ be a sequence of center leaves of $f$ so 
that $\Delta_H(f(W_i),W_{i+1})~<~\varepsilon$, for every $i$. We build  a converging sequence $\left\{W_{0,i}\right\}_{i\geq 0}$ of center leaves in $W^u(W_0)$ 
so that the limit leaf will shadow the pseudo-orbit $\left\{W_i\right\}_{i \geq 0}$. We build it by induction, see Figure~\ref{fig:shadowing} for illustration.  \\

Before we state the exact induction we construct the first center leaves
$W_{0,1}$ and $W_{0,2}$ to motivate the inductive assumption.

By definition of  $\Delta_H(f(W_0),W_1)$, we can find lifts $\widetilde{f(W_0)}$ and $\widetilde{W_1}$ 
such that the minimal Hausdorff distance is attained, 
that is $d_H(\widetilde{f(W_0)},\widetilde{W_1})= \Delta_H(f(W_0),W_1)~<~\varepsilon$. 

By the choice of $\varepsilon$ we find then a unique center leaf 
$\widetilde{f(W_{0,1})}~=~W^u_{\mu_1}(\widetilde{f(W_0)})\cap W^s_{\mu_1}(\widetilde{W_1})$.
Inequality~(\ref{e.distances}) implies with $\widetilde{f(W_{0,1})} \subset W^u_{\mu_1}(\widetilde{f(W_0)})$ and $\widetilde{f(W_{0,1})} \subset W^s_{\mu_1}(\widetilde{W_1})$ the following two estimates
\begin{align}\label{e.deltas}
\Delta^u_{proj}(\widetilde{f(W_0)},\widetilde{f(W_{0,1})})\leq Cd_H(\widetilde{f(W_0)},\widetilde{f(W_{0,1})})
&< C\varepsilon \mbox{, and }\nonumber\\ 
\Delta^s_{proj}(\widetilde{W_1},\widetilde{f(W_{0,1})})\leq Cd_H(\widetilde{W_1},\widetilde{f(W_{0,1})})&< C\varepsilon.
\end{align} 
Consider points $\widetilde{f(x)}~\in~\widetilde{f(W_0)}$ and 
$\widetilde{f(y)}~\in~\widetilde{f(W_{0,1})}$ 
so that
$\widetilde{f(y)}~\in~W^u_{\mu_1}(\widetilde{f(x)})$: the uniqueness of 
the projection says that $\widetilde{f(y)}$ is unique in this local unstable manifold and 
$d^u(\widetilde{f(x)}, \widetilde{f(y)})~<~C\varepsilon$.

Let $f(x)$ and $f(y)$ be the projections of these points, and consider $x~\in~W_0$ and $y~\in~W_{0,1}$.
The distance $d(x,y)$ is bounded by $\alpha d(f(x),f(y))~<~C\varepsilon\alpha~<~\delta_0$.  Therefore,
any lift of $W_0$ 
determines a lift for $W_{0,1}$.
According to Lemma~\ref{l.expansion}, applied to $f^{-1}$ and  the leaves 
$\widetilde{f(W_0)}$, and $\widetilde{f(W_{0,1})}$,  one gets: 
\begin{equation}\label{e.estimation1}
\Delta^u_{proj}(\widetilde{W_0},\widetilde{W_{0,1}})\leq 
 \alpha \Delta^u_{proj}(\widetilde{f(W_0)},\widetilde{f(W_{0,1})}) < C\varepsilon\alpha.
\end{equation}
Further, we have 
\begin{align}
\Delta_H(W_1,f(W_{0,1})) & \leq d_H(\widetilde{W_1}, \widetilde{f(W_{0,1})}) \; 
\mbox{(by Corollary~\ref{corol_distance})} \nonumber\\
&\leq  \Delta^u_{proj}(\widetilde{f(W_0)},\widetilde{f(W_{0,1})})+
\Delta^s_{proj}(\widetilde{W_1},\widetilde{f(W_{0,1})})\; \mbox{(by Lemma~\ref{l.metric})}\nonumber\\
&= \Delta^s_{proj}(\widetilde{W_1},\widetilde{f(W_{0,1})})\;(\mbox{as }\;\widetilde{f(W_{0,1})}\;\mbox{ belongs to the stable manifold of }\widetilde{W_1})\nonumber\\
&<C\varepsilon \;\mbox{(by inequality~(\ref{e.deltas}))}.
\end{align}
Now we start the construction of $W_{0,2}$.  For that, one needs to consider the image $f^2(W_{0,1})$. 

We fix points $\tilde{w}_1~\in~\widetilde{W_1}$ and $\tilde{w}_2~\in~\widetilde{f(W_{0,1})}$ 
such that $d(\tilde{w}_1,\tilde{w}_2)~<~C\varepsilon$.  
Then we consider the projections under the covering map.
All lengths are preserved as the covering map is a local isometry, 
so $d(w_1, w_2)~<~C\varepsilon$. 

Now we consider the images $f^2(W_{0,1}), f(W_1), f(w_1),f(w_2)$ under $f$ 
and the next center leaf $W_2$ of the pseudo orbit. 
We choose lifts $\widetilde{f(W_1)}$ and $\widetilde{W_2}$ to the holonomy 
cover such that the minimal Hausdorff distance is attained,  so that  
$d_H(\widetilde{f(W_1)},\widetilde{W_2})~<~\varepsilon.$ 

Further, we choose the lift $\widetilde{f(w_1)}$ which lies in $\widetilde{f(W_1)}$.
There is a unique lift $\widetilde{f(w_2)}$ such that 
$$d(\widetilde{f(w_1)},\widetilde{f(w_2)})=d(f(w_1),f(w_2)).$$

Then we pick the lifted center leaf $\widetilde{f^2(W_{0,1})}$ which contains $\widetilde{f(w_2)}$. 
According to Lemma~\ref{l.expansion},  $\widetilde{f^2(W_{0,1})}$ is contained in the local stable manifold of $\widetilde{f(W_1)}$ 
and we have with inequality (\ref{e.estimation1})
\begin{equation}\label{e.estimation2}
d_H(\widetilde{f(W_1)}, \widetilde{f^2(W_{0,1})})\leq 
\Delta^s_{proj}(\widetilde{f(W_1)}, \widetilde{f^2(W_{0,1})}) < C\alpha\varepsilon < \varepsilon
\end{equation}
as $\alpha$ is chosen such that $C\alpha~<~1$.

This implies 
\begin{equation}\label{e.estimation3}
\Delta_H({f^2(W_{0,1})},{W_2})\leq d_H(\widetilde{f^2(W_{0,1})},\widetilde{W_2}) < (1+\alpha C)\varepsilon < 2\varepsilon.
\end{equation}
By the choice of $\varepsilon$ there exists a unique center 
leaf
$$\widetilde{f^2(W_{0,2})} = W^u_{\mu_1}(\widetilde{f^2(W_{0,1})}) \cap W^s_{\mu_1}(\widetilde{W_2}).$$ 

Thanks to inequalities (\ref{e.estimation2}) and (\ref{e.estimation3}) one gets  
\begin{align*}
\Delta^u_{proj}(\widetilde{f^2(W_{0,1})},\widetilde{f^2(W_{0,2})}) &< C\varepsilon(1 + 2C\alpha) < 2C\varepsilon\\
\Delta^s_{proj}(\widetilde{W_2},\widetilde{f^2(W_{0,2})}) &< C\varepsilon(1 + 2C\alpha) < 2C\varepsilon.
\end{align*} 

We consider the  leaf $W_{0,2}$  whose image by $f^2$ is the projection of $\widetilde{f^2(W_{0,2})}$.
Recall that $f^2(W_{0,2})$ is contained in the local unstable manifold of $f^2(W_{0,1})$. 
Thus, applying Lemma~\ref{l.expansion}, there is a lift $\widetilde{W_{0,2}}$ of $W_{0,2}$ so that 
\begin{equation}\label{e.estimation4}
\Delta^u_{proj}(\widetilde{W_{0,2}},\widetilde{W_{0,1}})
<\alpha^2 \Delta^u_{proj}(\widetilde{f^2(W_{0,2}}),\widetilde{f^2(W_{0,1}}))
< 2C\varepsilon\alpha^2.
\end{equation}
With the inequalities (\ref{e.estimation1}) and (\ref{e.estimation4}) one gets
\begin{equation}\label{e.estimation5}
\Delta_{H}({W_{0,2}},{W_0})
\leq \Delta^u_{proj}(\widetilde{W_{0,2}},\widetilde{W_0})
< C\varepsilon\alpha + 2C\varepsilon\alpha^2 
< 2C\varepsilon \sum_{k=0}^2 \alpha^k.
\end{equation}

In this way we have built the first two elements $W_{0,1}$ and $W_{0,2}$ 
in the center unstable leaf $W^u(W_0)$.
During this construction the following inductive statement got 
visible which we have to prove in order to conclude the construction of the sequence $(W_{0,i})_i$:
\begin{lemma}\label{l.inductive} For any $i~\geq~0$ there is $W_{0,i}$  so that 
\begin{description}
\item[(I)] $\Delta_H(f^{i+1}(W_{0,{i}}),W_{i+1})~<~2\varepsilon,$
\item[(II)] for all $0~\leq~j~\leq~i$ one has 
$$\Delta_H(f^j(W_{0,i}),W_j)< 2C\varepsilon \sum_{k=0}^{i-j} \alpha^k <\eta.$$
 
\end{description}

\end{lemma}
\begin{proof}[Proof of Lemma~\ref{l.inductive}]
For $i~=~0$ we set $W_0~=~W_{0,0}$. For $i~=~1$ we have already proved properties (I) and (II) above. 
It remains to prove the inductive step: we assume Lemma~\ref{l.inductive} is done for $0,\dots, i$ 
and we build $W_{0,i+1}$ satisfying the announced inequalities. 

According to (I) we have 
$$\Delta_H(f^{i+1}(W_{0,i}), W_{i+1})<2\varepsilon.$$
So there are lifts $\widetilde{W_{i+1}}$ and $\widetilde{f^{i+1}(W_{0,i})}$ whose Hausdorff 
distance realizes 
the infimum, and is therefore bounded by $2\varepsilon$. 

Consequently, we can construct $W_{0,i+1}$ with the help of the unique center leaf  
$$\widetilde{f^{i+1}(W_{0,i+1})} := 
W^u_{\mu_1}(\widetilde{f^{i+1}(W_{0,i})})\cap W^s_{\mu_1}(\widetilde{W_{i+1}})$$ which satisfies
\begin{align}\label{e.deltas_i}
\Delta^u_{proj}(\widetilde{f^{i+1}(W_{0,i+1})}, \widetilde{f^{i+1}(W_{0,i})})&< 2C\varepsilon \mbox{ and }\\
 \Delta^s_{proj}(\widetilde{f^{i+1}(W_{0,i+1})}, \widetilde{W_{i+1}})&<2C\varepsilon.
\end{align}
Then -- as a consequence of  Claim~\ref{c.projection}  -- $W_{0,i+1}$ is the center leaf whose image by $f^{i+1}$ is the projection of 
$\widetilde{f^{i+1}(W_{0,i+1})}$.  Let us show that $W_{0,i+1}$ satisfies the inequalites announced in (I) and (II). 
We prove inequality (I) of Lemma~\ref{l.inductive} by the following claim:
\begin{claim}\label{c.firstitem} $\Delta_H(f^{i+2}(W_{0,i+1}),W_{i+2})< 2\varepsilon$.
\end{claim}
\begin{proof}[Proof of Claim~\ref{c.firstitem}] 
Consider lifts $\widetilde{f(W_{i+1})}$ and $\widetilde{W_{i+2}}$ such that the minimal Hausdorff 
distance in the holonomy cover is attained, so 
$$d_H(\widetilde{f(W_{i+1})}, \widetilde{W_{i+2}}) < \varepsilon.$$ 

The center leaf $\widetilde{f^{i+1}(W_{0,i+1})}$ lies by construction in the local center stable leaf 
of $\widetilde{W_{i+1}}$, so we have $\Delta^u_{proj}(\widetilde{f^{i+1}(W_{0,i+1})},\widetilde{W_{i+1}})=0$.
According to Lemma~\ref{l.expansion}, there is a lift of $f(f^{i+1}(W_{0,i+1}))$ so that 
\begin{align*}
 d_H(\widetilde{f^{i+2}(W_{0,i+1})},\widetilde{f(W_{i+1})})
 &\leq\Delta^s_{proj}(\widetilde{f^{i+2}(W_{0,i})},\widetilde{f(W_{i+1})})\;\\
 &(\mbox{by Lemma~\ref{l.metric} and}\;\Delta^u_{proj}(\widetilde{f^{i+1}(W_{0,i+1})},\widetilde{W_{i+1}})=0)\\
&\leq \alpha \Delta^s_{proj}(\widetilde{f^{i+1}(W_{0,i+1})}, \widetilde{W_{i+1}}) \;
(\mbox{by Lemma~\ref{l.expansion}})\\
&< 2C\varepsilon\alpha\;(\mbox{by inequality~(\ref{e.deltas_i}}))\\
&< \varepsilon.
\end{align*}

The triangle inequality implies 
$$ \Delta_H({f^{i+2}(W_{0,i})},{W_{i+2}})\leq d_H(\widetilde{f^{i+2}(W_{0,i})},\widetilde{W_{i+2}})<2\varepsilon, $$
which proves the claim. 
\end{proof}
The claim proved the first inequality (I) of the lemma. Let us now prove (II):
\begin{claim}\label{c.seconditem}
For all $0~\leq~j~\leq~i+1$ of Lemma~\ref{l.inductive} one has 
$$\Delta_H(f^j(W_{0,i+1}),W_j)< 2C\varepsilon \sum_{k=0}^{i+1-j} \alpha^k <\eta.$$
 
\end{claim}
\begin{proof}[Proof of Claim~\ref{c.seconditem}]

For any $0~\leq~j~\leq~i$, the induction hypothesis implies that  
there are lifts $\widetilde{f^j(W_{0,i})}$ and $\widetilde{W_j}$ such that 
$$d_H(\widetilde{f^j(W_{0,i})},\widetilde{W_j}) =\Delta_H(f^j(W_{0,i}),W_j)< 2C\varepsilon \sum_{k=0}^{i-j} \alpha^k.$$

On the other hand one has that $\widetilde{f^{i+1}(W_{0,i+1})}$ is contained in the unstable manifold of 
$\widetilde{f^{i+1}(W_{0,i})}$ and - due to inequality~(\ref{e.deltas_i}) - 
$\Delta^u_{proj}(\widetilde{f^{i+1}(W_{0,i+1})},\widetilde{f^{i+1}(W_{0,i})} )~<~2C\varepsilon$.

According to Lemma~\ref{l.expansion} one gets that there is a lift  $\widetilde{f^j(W_{0,i+1})}$ of 
$f^j(W_{0,i+1})$ so that 
\begin{align*}
\Delta_{H}(f^{j}(W_{0,i+1}),f^{j}(W_{0,i}) )
&\leq \Delta^u_{proj}(\widetilde{f^{j}(W_{0,i+1})},\widetilde{f^{j}(W_{0,i})} )\\
&\leq \alpha^{i+1-j}\Delta^u_{proj}(\widetilde{f^{i+1}(W_{0,i+1})},\widetilde{f^{i+1}(W_{0,i})} )\\
&<2C\alpha^{i+1-j}\varepsilon.
\end{align*}

As a consequence the triangle inequality implies: 
\begin{align*} 
\Delta_{H}(f^{j}(W_{0,i+1}),W_j )
&\leq 2C\varepsilon \sum_{k=0}^{i-j} \alpha^k+2C\varepsilon\alpha^{i+1-j}\\
&=2C\varepsilon \sum_{k=0}^{i+1-j} \alpha^k.
\end{align*}
This gives the announced inequality for $0~\leq~j~\leq~i$. 

For $j~=~i+1$ the announced inequality says 
$\Delta_{H}(f^{i+1}(W_{0,i+1}),W_{i+1} )~\leq~2C\varepsilon$.  
This comes from the fact that - according to inequality~(\ref{e.deltas_i}) - 
$$\Delta_{H}(f^{i+1}(W_{0,i+1}),W_{i+1} )
\leq \Delta^s_{proj}(\widetilde{f^{i+1}(W_{0,i+1})},\widetilde{W_{i+1}} )\leq 2C\varepsilon.$$
This finishes the proof of the claim.
\end{proof}
The proof of Claim~\ref{c.seconditem} gives us (II) of Lemma~\ref{l.inductive} which finishes its proof.
 
\end{proof}

Let $W_{\infty}$ be the limit point of the sequence $(W_{0,i})_{i\geq 0}$, then for every fixed $j~\geq~0$ item (II) of Lemma~\ref{l.inductive} implies 
$$\Delta_H(f^j(W_{\infty}),W_j) = \lim_{i \rightarrow \infty}\Delta_H(f^j(W_{0,i}),W_j) < 2C\varepsilon \frac{1}{1-\alpha} < \eta.$$
Therefore $W_{\infty}$ is the center leaf whose orbit stays $\eta$-close to the positive pseudo orbit $(W_j)_{j \geq 0}$, i.e. it is its shadowing orbit. This concludes the proof of Proposition~\ref{p.shadowing}.

\end{proof}

\subsection{Uniqueness of the shadowing and existence of periodic center leaves}

If the center foliation has non-trivial holonomy, there are in general an uncountable set of 
center-leaves $\delta$-shadowing the same $\varepsilon$-pseudo orbit $\left\{W_i\right\}_i$ of center leaves. Therefore the statement of 
Theorem~\ref{shadowing_lemma} is not enough for ensuring the existence of periodic center leaves 
shadowing the periodic pseudo orbits.

However the lack of uniqueness of the shadowing comes from the choice of an appropriate lift we have to do during the construction of the shadowing orbit.
The idea of this section is to enhance the pseudo orbits by the possible choices, 
and then to recover the uniqueness in that setting: 

\begin{otherthm}\label{t.unique} Let $f$ be a partially hyperbolic diffeomorphism with a
uniformly compact invariant center foliation. 
Let $\left\{U_j,p_j\right\}_j$ be a holonomy cover and $\delta_0$ (as in Lemma~\ref{lemma_independence_cover}). 

Then there is $\eta_0$  so that for any $0<\eta<\eta_0$ there is $\varepsilon>0$ such that
\begin{itemize}
 \item for any $\varepsilon$-pseudo orbit $\left\{W_i\right\}_{i\in\mathbb{Z}}$ 
 for the modified Hausdorff distance $\Delta_H$
 \item for any $i\in\mathbb{Z}$ and any $j$ so that the $B_H(W_i,\delta_0)\subset U_j$,
 \item for any lifts of $W_i$ and $f(W_{i-1})$ with 
 $d_H(\widetilde{W_i},\widetilde{f(W_{i-1})})<\varepsilon$,
\end{itemize}

there is a unique center leaf $W$ admitting lifts $\widetilde{f^i(W)}$ so that 
\begin{itemize}
 \item $\Delta_H(f^i(W)), W_i)<\eta$
 \item $d_H(\widetilde{f^i(W))}, \widetilde{W_i})<\eta$
 \item $\widetilde{f^{i+1}(W))}$ is the lift of $f(f^{i}(W))$ associated to the pair of lifts
 $(\widetilde{f^i(W))}, \widetilde{W_i})$ and the lift $\widetilde{f(W_i)}$ by Remark~\ref{uniquelift}.
\end{itemize}
\end{otherthm}
The proof is indeed the proof we wrote for Theorem~\ref{shadowing_lemma}. 

We are now ready for proving Theorem~\ref{t.periodic} which states that the periodic center leaves are dense 
in the set of center leaves which are chain recurrent for the quotient dynamics

\begin{proof}[Proof of Theorem~\ref{t.periodic}]It is enough to prove that any periodic $\varepsilon$-pseudo orbit
of center leaves is $\eta$-shadowed by (at least one) a periodic center leaf. 
For that, just consider a periodic choice of lifts of the pseudo orbit in Theorem~\ref{t.unique}.  
\end{proof}

\section{Plaque expansivity}\label{sec:plaque}

The aim of this section is to prove the plaque expansivity. Let us recall the precise definition. 

\begin{defn}
Let $f$ be a partially hyperbolic diffeomorphism with an $f$-invariant center foliation $\cF^c$.
We call $(f,\cF^c)$ \emph{plaque expansive} if there is $\eta~>~0$ with the
following property:
consider  $\eta$-pseudo orbits $(x_i),(y_i)$ preserving plaques of the center foliation, i.e. 
$f(x_i)$ lies in the same plaque as $x_{i+1}$ with $d(f(x_i),x_{i+1})~<~\eta$ and $f(y_i)$ in the plaque of $y_{i+1}$, $d(f(y_i),y_{i+1})~<~\eta$.
Assume that $d(x_i,y_i)~<~\eta$ for every $n~\in~\ZZ$. Then $x_0$ lies in the same center plaque as $y_0$.
\end{defn}
\begin{rem}
\cite{HPS70} proves that every $C^1$-center foliation is plaque expansive. 
\end{rem}
The aim of this section is to prove Theorem~\ref{plaquex}:  any invariant uniformly compact center foliation is plaque expansive.
\begin{figure}
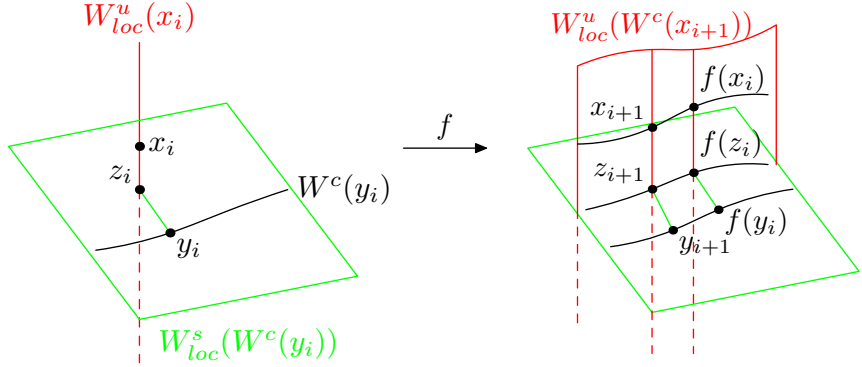

\begin{subfigure}{0.4\textwidth}
\includegraphics[width=\textwidth]{plaque.mps}
\end{subfigure}
\begin{subfigure}{0.4\textwidth}
\includegraphics[width=1.2\textwidth]{plaque_iterate.mps}
\end{subfigure}

\caption{Two pseudo-orbits $(x_i)_i,(y_i)_i$ and construction of the pseudo-orbit $(z_i)_i$ where $z_i=W^u_{loc}(x_i)\cap W^s_{loc}(W^c(y_i)$ is the projection of $x_i$ on the local center stable manifold of $y_i$.}\label{fig:plaque}
\label{fig:plaque}
\end{figure}
We consider some very small $\eta~>~0$ (we will define $\eta$ later) and $\eta$-pseudo-orbits $(x_i),(y_i)$ with central jumps, and remaining at a distance $d(x_i,y_i)$ less than $\eta$, see Figure~\ref{fig:plaque} for an illustration and the construction of the pseudo orbit $(z_i)$. 
We assume that $y_0$ does not belong to the center stable manifold of $x_0$.  

First of all,
\begin{lemma}\label{l.pseudoproj}\label{l.pseudoplaque}
 let $z_i~\in~W^u_{loc}(x_i)$ be the projection of $x_i$ on the local center stable manifold of $y_i$.
 Then the sequence $(z_i)$ is also a pseudo orbit with jumps in the center leaves, 
 remaining close to the sequences $(x_i)$ and $(y_i)$. 
 
 Analogously, let $w_i~\in~W^s_{loc}(y_i)$ be the projection of $y_i$ 
 on the local center unstable manifold of $x_i$. Then $(w_i)$ is a pseudo-orbit with 
 jumps in the center leaves, remaining close to $(x_i)$, $(y_i)$, and therefore to $(z_i)$. 
 
 Furthermore, for any $i$, the points $z_i$ and $w_i$ belong to the same 
center plaque. 
 
\end{lemma}
\begin{proof} We need to prove that $z_{i+1}$ is on the same center leaf as $f(z_i)$.   
This is due to the dynamical coherence: the local unstable manifolds through the center
plaque containing 
$x_{i+1}$ and $f(x_i)$ intersects the local stable manifold through the center plaque containing
$y_{i+1}$ and
$f(y_i)$ in a center plaque containing $z_{i+1}$ and a point which is the intersection 
of the local 
unstable of $f(x_i)$ with the local center stable of $f(y_i)$.  
Notice that these local invariant manifolds 
are the image of the local invariant manifolds (with size multiplied by a bounded factor (bounded by 
$\|Df\|$)).  Therefore, these local invariant manifolds intersect on $f(z_i)$. 
Thus $f(z_i)$ belongs to 
the same center plaque as $z_{i+1}$. It works in the same way for $(w_i)$.

Let us now prove the last claim of the lemma.
 There is $u_i$ in the center plaque of $y_i$ so that $z_i~\in~W^s_{\mu_1}(u_i)$.  Therefore, $z_i$ is 
 the projection of $u_i$ on the center unstable manifold of $x_i$. 
 Consider a small path $u_{i,t}$ joining $u_i$ to $y_i$ in the  center plaque of $y_i$.  Then the projection of 
 $u_{i,t}$ on the center unstable plaque of $x_i$ defines a small path in the intersection of the local 
 center stable manifold of $y_i$ with the local center unstable manifold of $x_i$. Therefore this small path is contained 
 in a plaque of a center leaf. Thus $w_i$ to $z_i$ are joined by a path inside a center plaque.
\end{proof}

Now, Theorem~\ref{plaquex} is a direct consequence of the following proposition, where $\mu >0$ is given by Lemma~\ref{l.Delta-expansion}:
\begin{prop}\label{p.plaquex} \begin{itemize}
\item For any pair of $\mu$-pseudo orbits $(x_i)$, $(z_i)$  
preserving the center plaques, and such that $z_i~\in~W^u_{\mu}(x_i)$ for any $i$, one has 
$x_i=z_i$ for any $i$.

\item For any pair of $\mu$-pseudo-orbits $(y_i)$, $(w_i)$  
preserving the center plaques, and such that $w_i~\in~W^s_{\mu}(y_i)$ for any $i$, one has 
$y_i=w_i$ for any $i$.
\end{itemize}
\end{prop}
The two items of the proposition are equivalent (up to changing $f$ by $f^{-1}$) 
so that we prove only the first item.

Proposition~\ref{p.plaquex} is a straightforward consequence of the next lemma:
\begin{lemma}\label{l.plaquex1}\label{l.plaquex2}
With the notation of Proposition~\ref{p.plaquex}: 
\begin{itemize}
 \item there is $N~>~0$ so that for any $n~\geq~N$ one has $x_n=z_n$.
 \item For any integer $n$, if $x_n=z_n$ then $x_{n-1}=z_{n-1}$. 
\end{itemize}
\end{lemma}
\begin{proof}[Proof of Lemma~\ref{l.plaquex1}] 
By the definition of pseudo orbits preserving the center plaques, the points $f^{-n}(x_n)$ and $f^{-n}(z_n)$ 
belong to the center leaf of $x_0$ and of $z_0$, respectively. Furthermore, as $z_n~\in~W^u_\mu(x_n)$, 
for any $n~>~0$ one gets  $f^{-n}(z_n)~\in~W^u_{\alpha^n\mu} (x_n)$ (recall that $\alpha~<~1$ is the 
contraction rate in the stable bundle for an adapted metric). In particular, their distance tends to $0$. 
As $W^c(x_0)$ and $W^c(z_0)$ are compact manifolds which are either disjoint or equal, this implies that 
$W^c(x_0)=W^c(z_0)$.

As $W^c(x_0)$ is a compact manifold transverse to the unstable foliation (indeed, in general position), this implies that there is 
$\eta$ (depending on $W^c(x_0)$) so that every local  unstable leaf $W^u_\eta(x)$, $x\in W^c(x_0)$ intersects 
$W^c(x_0)$ only at the point $x$. For $\alpha^n\mu~<~\eta$ this implies $x_n=z_n$ concluding the proof of the first 
claim of the lemma.

We now prove the second claim. The proof argues recursively. Recall that $\mu_2 > 0$ is given by Lemma~\ref{l.proj} 
and guarantees that the unstable projection is unique for points at a distance $<\mu_2$. 
Assume $x_n=z_n$.  Consider $x_{n-1}$ and $z_{n-1}~\in~W^u_{\mu}(x_{n-1})$.

By assumption, $f(x_{n-1})$ belongs to $W^c_\mu(x_n)$ and $f(z_{n-1})~\in~W^c_{\mu}(z_n)=W^c_\mu(x_n)$, 
(as $z_n=x_n$).

Furthermore, $z_{n-1}~\in~W^u_\mu(x_{n-1})$ so that $f(z_{n-1})~\in~W^u_{\lambda\mu}(f(x_{n-1}))$ where 
$\lambda$ is a bound of $\|Df\|$. For $\lambda\mu$ smaller than $\mu_2$, 
the intersection of a local unstable manifold with a local center manifold consist in at most one point,
so that $f(x_{n-1})=f(z_{n-1})$ that is $x_{n-1}=z_{n-1}$. 
\end{proof}

\begin{rem}The proof of the first claim shows indeed that $W^c(x_i)=W^c(z_i)$ for any $i$, that is, $x_i$ and $z_i$ lie in the same center leaf.  
However, it does not 
show directly that $x_i$ and $z_i$ are equal or belong to the same center plaque, i.e. the same local center leaf, explaining 
why we need the second claim.
\end{rem}

\section{Non-compactness of center unstable and center stable leaves.}\label{sec:noncompact}
Under the assumption of a uniformly compact invariant center foliation, 
we  proved the dynamical coherence (existence of center stable  
and center unstable foliations denoted by  $\mathcal{F}^{cu}$ and  $\mathcal{F}^{cs}$, respectively). 

The aim of this section is to start the topological study of these foliations: 

\begin{otherthm}
Let $f:M~\rightarrow~M$ be a partially hyperbolic $C^1$-diffeomorphism 
with an $f$-invariant uniformly compact center foliation. 
Then every leaf of the center unstable foliation $\mathcal{F}^{cu}$ is non-compact.
The same holds for every leaf of the center stable foliation $\mathcal{F}^{cs}$.
\label{theorem_noncompact} 
\end{otherthm}
This result has been announced in \cite{C11}.   We include here a detailled proof because the proof is easy and
this result is used in \cite{B12}, and \cite{C11} is still unpublished. 

\begin{lemma}\label{l.intersection}
Let $f:M~\rightarrow~M$ be a partially hyperbolic $C^1$-diffeomorphism 
with an $f$-invariant uniformly compact center foliation. Let $n$ be the maximal (finite) order of the holonomy group of center leaves. 
Then for every $y,z~\in~M$ we have $$\sharp \left\{W^s (y) \cap W^c(z)\right\}~<~n.$$ 
\end{lemma}
\begin{proof}
Notice that any local stable  manifold of size $\mu$ admits a lift in a holonomy cover $\left\{U_i, p_i\right\}$ 
and the lifted local stable manifold intersects each lifted center leaf in at most one point.  
The preimage by $p_i$ of any center manifold consists of at most $\max \left|\Hol\right| = \max_{L \in \cF^c}\left|\Hol(L)\right|$ points:
this implies that for any 
$y,z \in M$ we have $$\sharp \left\{W^s_\mu (y) \cap W^c(z)\right\}  \leq n=\max \left|\Hol\right|.$$ 

Assume now, arguing by contradiction,  that there is $y,z$ so that the whole stable manifold $W^s(y)$ cuts $W^c(z)$ in 
strictly more than $n$ points.  Then there is $r>0$ so that the stable manifold $W^s_r(y)$ of size $r$  
cuts $W^c(z)$ in more than $n$ points.  Choose $i~>~0$ so that $\alpha^ir~<~\mu$, 
where $\alpha$ is the contraction rate of the stable bundle.

Then $W^s_{\alpha^ir}(f^i(y))~\subset~W^s_\mu(f^i(y))$ cuts $W^c(f^i(z))$ in more than $n$ points, 
leading to a contradiction. 
  
\end{proof}

For the proof of Theorem~\ref{theorem_noncompact} we notice the following fact: 

\begin{lemma}\label{l.noncompact}
Let $V$ be a $C^1$-compact manifold admitting a continuous splitting $TV=F\oplus G$ in transverse subbundles.
Assume that there are 
$C^0$-foliations $\mathcal{F}$ and  $\mathcal{G}$ with $C^1$-leaves tangent to $F$ and $G$, respectively. 
Then for any non-compact leaf $L_G$ of $\mathcal{G}$, there is a leaf $L_F$ of $\mathcal{F}$ so that 
the intersection
$L_F\cap L_G$ is infinite. 
\end{lemma}
\begin{proof} As $L_G$ is assumed to be non-compact and $V$ is compact, there is a point $x~\in~V$
which is contained in the limit of an infinite sequence of plaques of $L_G$. Therefore, 
the leaf $L_F(x)$ cuts $L_G$ in infinitely many points.
\end{proof}
Now we can prove Theorem~\ref{theorem_noncompact} straightforwardly: 
\begin{proof}[Proof of Theorem~\ref{theorem_noncompact}]
Assume that there exists a compact leaf $W^{cs}\in \cF^{cs}$. Therefore, we can apply Lemma~\ref{l.noncompact} to $TW^{cs}=E^{c}|_{W^{cs}}\oplus E^s|_{W^{cs}}$. 
Every stable leaf $W^s(z)$ for $z \in W^{cs}$ is non-compact. 
Accordingly, for any leaf $W^s(z)$ there exists a center leaf 
$W^c\in W^{cs}$ which intersects $W^{s}(z)$ infinitely many times contradicting 
Lemma~\ref{l.intersection}. Hence, every center stable leaf is non-compact. 
A similar argument proves that every center unstable leaf is non-compact, finishing the proof. 
\end{proof}

\section{Comments, examples and open questions}\label{sec:comments}

In this section we first present a very simple example from \cite[Section 4.1.3]{BoW05} which illustrates many of the 
pathological behavior one can get for the dynamics on the center leaf space. Then we will ask some questions on the existence of periodic center leaves and on the dynamical characterization of the stable manifolds - in the light provided by this example.

\subsection{Non-expansivity and infinite non-uniqueness of the shadowing}\label{s.nonexpansive}

Let us come back to an example presented in \cite[Section 4.1.3]{BoW05}.
\begin{itemize}
 \item Let $A~\in~SL(2,\ZZ)$ be an Anosov matrix, 
considered as a diffeomorphism of the torus $\mathbb{T}^2=\mathbb{R}^2/\mathbb{Z}^2$.  
\item Let $M$ be the compact $3$-manifold obtained as the quotient of
$\mathbb{T}^3=\mathbb{T}^2\times \mathbb{S}^1$ by the free involution 
$\varphi\colon(r,s,t)~\mapsto~(-r,-s,t+\frac12)$.
\item The involution $\varphi$ commutes with the diffeomorphism $\tilde f= A\times \id_{\mathbb{S}^1}$ so that
$\tilde f$  passes to the quotient as a diffeomorphism $f$ of $M$.
\end{itemize}
The diffeomorphism $f$ is partially hyperbolic, its center foliation is a circle 
foliation which is a Seifert bundle over the sphere $S^2$, with $4$ exceptional 
leaves which have a non-trivial holonomy, which is indeed generated by $-\id$. 

The space of center leaves is therefore the sphere $S^2$ naturally endowed with an 
orbifold structure with $4$ singular points. The stable and unstable foliations of $f$ induce on the quotient
transverse singular foliations, with exactly one singularity at each singular point of the orbifold. 
These singularities are \emph{one prong singularities of generalized pseudo Anosov type}. 

This implies that for every $\varepsilon~>~0$, for every center leaf $W^c$ close to one of 
the exceptional fiber, $W^s_{\varepsilon}(W^c)\cap W^u_{\varepsilon}(W^c)$ consist in the union of 
$W^c$ and another center leaf $W^c_1$, see Figure~\ref{fig:example}.  

As a consequence, for every $n~\in~\ZZ$ the distance $d_H(f^n(W^c),f^n(W^c_1))$ is bounded by $\varepsilon$.  

\begin{rem}
 The argument above shows that the quotient dynamics $A/-\id$ induced by $f$ 
on the center leaves spaces $S^2$ is not expansive. This non-expansivity is indeed expected as there are no expansive 
homeomorphism one the $2$-sphere (see \cite{L89}).
\end{rem}
\begin{figure}
\includegraphics[width=0.5\textwidth]{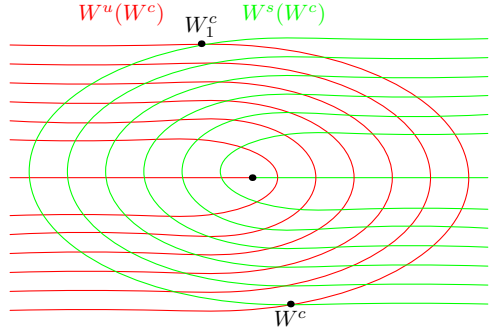}
\caption{Intersection of local stable and unstable leaves in a neighborhood of a one prong singularity on the quotient space of the center leaves.}
\label{fig:example}
\end{figure}
Let us investigate a little bit further this lack of expansivity: 

The Anosov diffeomorphism $A$ admits a dense subset of periodic points.  
One deduces that there are periodic center leaves 
$W^c$ which are arbitrarily close to an exceptional 
leaf. In other words, one may assume that, in the comment above, $W^c$ 
is a periodic center leaf, and 
$W^c_1$ is a homoclinic intersection associated to $W^c$. 

One deduces that the maximal invariant set of any small neighborhood of 
$W^c\cup W^c_1$ contains an 
invariant compact set, saturated by the center foliation, and which 
projects onto $S^2$ as 
a non-trivial hyperbolic basic set of $A/-\id$. This shows that

\begin{lemma}
For any $\varepsilon~>~0$ there is a invariant compact set $K$ saturated for the center foliations, 
homeomorphic to the product of a Cantor set by $S^1$, and so that, for any leaves $W_1,W_2~\subset~K$ one has 
$d_H(f^n(W_1),f^n(W_2))~<~\varepsilon$.
\end{lemma}

\subsection{Dynamical classification of invariant manifolds}

For hyperbolic systems, the set of points whose positive orbits remain at a distance less 
than $\varepsilon~>~0$ of the orbit of $x$ is the local stable manifold of $x$.

This characterization is no more true for the dynamics induced on the center leaf space: 
in the example above, for any  $\varepsilon~>~0$, there are center leaves $W_1,W_2$ whose positive 
(and indeed negative) orbits remains at a distance less than $\varepsilon$ but which are not in the 
same stable manifold. 

\begin{prop}\label{p.example}
 Given an Anosov automorphism $A$ of 
 $\mathbb{T}^2$ the quotient $\varphi=A/-\id$  admits 
 points $x,y$ so that the distance $d(\varphi^n(x),\varphi^n(y))$ 
 tends to $0$ when $n~\to~\pm\infty$, 
 but the lifts of $x$ and $y$ on $\mathbb{T}^2$ do neither lie in the 
 same stable nor unstable manifold.
\end{prop}
As a direct corollary one gets:
\begin{corol}\label{c.example} Let $f$ be the partially hyperbolic diffeomorphisms associated to  an
Anosov automorphism $A$ of 
 $\mathbb{T}^2$ in Section~\ref{s.nonexpansive}.
Then, there exist leaves $W_1$, $W_2$ 
 with $d_H(f^n(W_1) f^n(W_2))$ tends to $0$ for $n~\rightarrow~\infty$, but
 $W_2~\nsubseteq~W^s(W_1)$. 
 \end{corol}

\begin{proof}[Proof of Proposition~\ref{p.example}] Fix some $\delta~>~0$ and $\varepsilon~>~0$ 
so that every pseudo orbit of the Anosov map $A$ 
is  uniquely $\delta$-shadowed by a true orbit of $A$. 
We will use the following property which comes directly from
the proof of the Shadowing  
Lemma: any $\delta$-pseudo orbit $(x_i)_{i\in\mathbb{Z}}$ whose 
jumps tend to $0$ 
(that is $d(f(x_i),x_{i+1})~\to~0$ for $|i|~\to~\infty$) is uniquely $\delta$-shadowed 
by the orbit of a point $x$ with 
$d(f^i(x),x_i)$ tending to $0$ with $|i|~\to~\infty$.

Consider a sequence of periodic points $p^+_n$, $n~\in~\ZZ$ 
tending to $0$ as $|n|~\to~\infty$ and so that $d(p^+_n, 0)~<~\varepsilon$. 
Denote $p^-_n=-\id(p^+_n)$: then $p^+_n$ and $p^-_n$ have the 
same period $\pi_n$, tending to $\infty$ as $|n|\to \infty$. 
Let us denote by 
$\gamma^\pm_n$ the orbit of $p^\pm_n$, with initial point at $p^\pm_n$.  
Notice that, for every $n,m$ one has $d(p^\pm_n,p^\pm_m)~<~\varepsilon$ 
and $d(p^\pm_n,p^\pm_m)~\to~0$ as $|n|,|m|~\to~\infty$.

For any sequence $\mu=\{\mu_i\in\{+,-\}, i\in\ZZ\}$  we consider the  
infinite sequence $\Gamma_\mu$ 
$$  \dots p^{\mu_{-1}}_{-1},\dots, A^{\pi_{-1}}(p^{\mu_{-1}}_{-1}), 
p^{\mu_0}_0,\dots, A^{\pi_0}(p^{\mu_0}_0), p^{\mu_1}_1,\dots, 
A^{\pi_1}(p^{\mu_1}_1)\dots $$
obtained by concatenation of the $\gamma^{\mu_n}_n$.

Then each of the $\Gamma_\mu= (\Ga_{\mu,i})_{i\in\ZZ}$ is a $\delta$-pseudo orbit whose jumps tends to $0$. 
Therefore, each ot the $\Gamma_\mu$ is $\delta$-shadowed by the orbit of a point 
$x_\mu$, and $d(f^i(x_\mu),\Gamma_{\mu,i})~\to~0$ with $i$.

Now the orbits of $p^+_n$ and of $p^-_n$ project onto the same orbit of $\varphi=A/-id$. 
Therefore all the pseudo orbits $\left\{\Gamma_\mu\,|\, \mu\in \{+,-\}^{\mathbb{Z}}\right\}$ of $A$ 
project onto the same pseudo orbit $\Gamma$ of $\varphi$.  Let $z_\mu$ be the 
projection of $x_\mu$.  Thus, $\Gamma$ is $\delta$-shadowed by the orbit of 
$z_\mu$ and $d(\varphi^i(z_\mu),\Gamma_i)~\to~0$ as $|i|~\to~\infty$. 

Consider now $\mu$, $\nu~\in~\{+,-\}^{\ZZ}$ so that both sets $\{ i | \nu_i=\mu_i\}$ 
and $\{j | \nu_j=-\mu_j\}$ are infinite and neither upper nor lower bounded.
For every $i$, denote 
$$d_i=\inf\{d(A^i(x_\mu), A^i(x_\nu)), d(A^i(-x_\mu), A^i(x_\nu))\}.$$ 
Our choice of $\mu$ and $\nu$ implies that $d_i$ does neither tends to $0$ as $i~\to~+\infty$ 
nor as $i\to - \infty$. Therefore, $x_\nu$ is neither in the stable nor 
in the unstable  manifold of  neither $x_\mu$ nor $-x_\mu$. 

Thus the orbits of $z_\mu$ and $z_\nu$ are positively and negatively asymptotic 
but are not in the same stable nor unstable leaf, ending the proof.

\end{proof}
\begin{rem} The proof of Proposition~\ref{p.example} gives a little bit more: given $x$ 
(corresponding in the proof to
the point $x_\mu$), the  points $y=x_\nu$ satisfying the conclusion of Proposition~\ref{p.example} are 
indeed uncountably many, corrsponding to all possible choices of $\nu$. 
 
\end{rem}

\subsection{Unique intergrability}

We have been unable to prove the uniqueness of the center foliation even with very strong hypotheses:
\begin{question}
 Let $f$ be a partially hyperbolic diffeomorphism admitting invariant uniformly compact center 
 foliations $\cW^1$ and $\cW^2$ (tangent to the same center-bundle). Are they equal? 
\end{question}

\bibliographystyle{amsalpha}
\bibliography{bibfile}
\end{document}